\documentclass[a4paper]{scrartcl}

\usepackage[utf8]{inputenc}

\usepackage{amsthm,amsmath,amsfonts,amssymb}
\numberwithin{equation}{section} 

\usepackage{mathrsfs}

\usepackage{tikz-cd}

\usepackage{mathtools}	
\mathtoolsset{showonlyrefs=true} 
\usepackage{bm}	
\usepackage{bbm}	

\usepackage[colorlinks=true]{hyperref}	
\usepackage[all]{hypcap}      

\usepackage[backend=bibtex,
			style=alphabetic,
			natbib=true,
			sorting=nty,
			doi=false,
			isbn=false,
			url=false,
			eprint=false,
		  ]{biblatex} 
\newcommand{\m}[1]{\mathcal{#1}}
\newcommand{\bb}[1]{\mathbb{#1}}

\newcommand{\mrm}[1]{\mathrm{#1}}

\newcommand{\f}[1]{\mathfrak{#1}}
\renewcommand{\Re}{\mathrm{Re}}
\renewcommand{\Im}{\mathrm{Im}}

\newcommand{\del}{\partial}
\newcommand{\delbar}{\bar{\partial}} 

\newcommand{\I}{\mathrm{i}\mkern1mu}	

\DeclarePairedDelimiter\abs{\lvert}{\rvert}	
\DeclarePairedDelimiter\norm{\lVert}{\rVert}	
\DeclareMathOperator{\arccot}{arccot} 

\DeclarePairedDelimiter{\set}{\{}{\}}	
\newcommand{\tc}{\mathrel{}\mathclose{}\middle|\mathopen{}\mathrel{}}	

\newtheorem{thm}{Theorem}[section]

\newtheorem{prop}[thm]{Proposition}
\newtheorem{lemma}[thm]{Lemma}
\newtheorem{cor}[thm]{Corollary}
\newtheorem{conj}[thm]{Conjecture}

\theoremstyle{definition}
\newtheorem{definition}[thm]{Definition}
\newtheorem{rmk}[thm]{Remark} 

\theoremstyle{remark}

\title{A K-energy functional for complexified K\"ahler classes}
\author{Carlo Scarpa}
\date{\vspace{-1em}}

\begin{document}

\maketitle

\begin{abstract}
\noindent Given a complexified K\"ahler class on a compact K\"ahler manifold, we introduce a generalisation of the K-energy functional, defined on the space of complexified K\"ahler potentials, whose critical points are solutions of the scalar curvature equation with B-field introduced by Schlitzer and Stoppa. We prove that this extended K-energy is convex along geodesics in the space of almost calibrated potentials. As an application, we show that, at least in some notable cases, solutions of the scalar curvature equation with B-field are unique in their class, confirming the expectation that the scalar curvature equation with B-field identifies canonical representatives of complexified K\"ahler classes.
\end{abstract}


\section{Introduction and main results}

Let~$X$ be a compact complex manifold of complex dimension~$n$, and fix two cohomology classes~$\alpha,\beta\in H^{1,1}(X)\cap H^2(X,\bb{R})$. We assume that $\alpha$ is a K\"ahler class, while~$\beta$ can be any cohomology class. We will refer to $\beta$ as the \emph{B-field class}, and elements of~$\beta$ will be called \emph{B-fields}; the terminology comes from mirror symmetry, where one often considers classes of the form~$\alpha^{\bb{C}}\coloneqq\beta+\I\alpha\in H^{1,1}(X)$, called \emph{complexified K\"ahler classes}. We will assume throughout that the \emph{complexified volume}~$(\alpha^{\bb{C}})^n$ does not vanish.

We are interested in the problem of choosing \emph{canonical} representatives of the complexified K\"ahler class: this question arises naturally in mirror symmetry, usually for~$X$ Calabi-Yau or Fano, on which complexified K\"ahler forms~$\omega^{\bb{C}}\coloneqq B+\I\omega$ are used to define the Fukaya category~$\m{F}(X,\omega^{\bb{C}})$. Different choices of~$\omega^{\bb{C}}$ in~$\alpha^{\bb{C}}$ give rise to~$A_\infty$-equivalent Fukaya categories, but to study geometric properties of~$\m{F}(X,\omega^{\bb{C}})$ and its objects it is necessary to fix a representative of~$\alpha^{\bb{C}}$.

For~$\beta=0$, this reduces to the famous question of how to find a special representative of a given K\"ahler class on~$X$, a version of the \emph{Calabi problem}. This is done by imposing some curvature conditions on the K\"ahler metric, such as being K\"ahler-Einstein, having constant scalar curvature, or being an extremal point for the Calabi functional. These conditions pick out at most one K\"ahler metric in each class, up to biholomorphisms of~$X$. This fundamental result has been established, in various degrees of generality, in~\cite{BandoMabuchi, Donaldson_scalarcurvature_embeddingsI, Mabuchi_uniqueness, ChenTian_uniqueness, BermanBerndtsson_uniquness}.

To fix a choice of representative for~$\alpha^{\bb{C}}$ then it is natural to impose a similar curvature condition on a complexified K\"ahler form in~$\alpha^{\bb{C}}$. Schlitzer and Stoppa proposed, in~\cite{SchlitzerStoppa}, a generalisation of the constant scalar curvature equation for complexified K\"ahler forms. They consider, for a K\"ahler form~$\omega\in\alpha$ and a B-field~$B\in\beta$, the \emph{scalar curvature equation with B-field}
\begin{equation}\label{eq:scalar_dHYM_singola}
	s(\omega)=c+\gamma\frac{(B+\I\omega)^n}{\omega^n},
\end{equation}
where~$s(\omega)$ denotes the scalar curvature of the metric defined by~$\omega$, and~$c\in\bb{R}$,~$\gamma\in\bb{C}$ are constants. This condition is obtained by coupling the constant scalar curvature equation for K\"ahler metrics with the \emph{deformed Hermitian Yang-Mills equation}, through the formalism of infinite-dimensional moment maps. They also show that, in a regime known as the \emph{large volume limit},~\eqref{eq:scalar_dHYM_singola} recovers the system
\begin{equation}\label{eq:largevolume_system}
\begin{cases}
	\Delta B=0\\
	s(\omega)=c,
\end{cases}
\end{equation}
which if often considered in the mirror symmetry literature to fix representatives of~$\alpha$ and~$\beta$ in the large volume limit, particularly for $X$ Calabi-Yau where the second equation in \eqref{eq:largevolume_system} reduces to Ricci-flatness of $\omega$.

One of the reasons to consider~\eqref{eq:scalar_dHYM_singola}, rather than~\eqref{eq:largevolume_system}, is that the existence of solutions of~\eqref{eq:scalar_dHYM_singola} is expected to be equivalent to some algebro-geometric stability condition that should depend on the B-field class, while the existence of solutions of~\eqref{eq:largevolume_system} is clearly independent from $\beta$. We refer to~\cite{ScarpaStoppa_toric_dHYM} and~\cite{Stoppa_KstabMirror} for more precise considerations regarding algebraic stability notions and possible connections with special properties of the mirror manifold of~$(X,\alpha^{\bb{C}})$.
  
The main goal of the present work is to describe a variational framework in which to study~\eqref{eq:scalar_dHYM_singola}. We will focus in particular on the geometry of the space of solutions of~\eqref{eq:scalar_dHYM_singola}, giving additional evidence for the following expectation of~\cite{ScarpaStoppa_toric_dHYM}: \emph{solutions of~\eqref{eq:scalar_dHYM_singola}, if they exist, are unique in~$\alpha\times\beta$, up to pullback by an automorphism of~$X$}. In other words, the scalar curvature equation with~$B$-field gives a way to choose a canonical representative for the complexified K\"ahler class~$\alpha^{\bb{C}}$. The heart of the paper is Section~\ref{sec:complexified_Kenergy}, which is dedicated to the proof of the following
\begin{thm}\label{thm:complexified_Kenergy}
	There is a functional~$\tilde{\m{M}}:\alpha^{\bb{C}}\to\mathbb{R}$ whose Euler-Lagrange equation is~\eqref{eq:scalar_dHYM_singola}. Moreover,~$\tilde{\m{M}}$ is convex along geodesics in the space of \emph{almost calibrated representatives} of~$\alpha^{\bb{C}}$.
\end{thm}
We refer to Section~\ref{sec:complexified_Kenergy} for the precise definition of the space of almost calibrated representatives of a complexified K\"ahler class, which is inspired by a similar definition in~\cite{Collins_stabilitydHYM}. Briefly, these are complexified K\"ahler forms~$\omega^{\bb{C}}\in\alpha^{\bb{C}}$ for which
\begin{equation}\label{eq:almost_cal}
	\Re(\gamma(\omega^{\bb{C}})^n)>0.
\end{equation}
The almost calibrated condition \eqref{eq:almost_cal} is crucial for most of the arguments in this paper.

\smallskip

The functional $\tilde{\m{M}}$ is an analogue of the \emph{K-energy functional} used in the study of constant scalar curvature K\"ahler (cscK, from now on) metrics, and we call it the \emph{complexified K-energy} of~$(X,\alpha^{\bb{C}})$. The existence of such a functional is expected from the moment map description of~\eqref{eq:scalar_dHYM_singola}, and it can be interpreted as an infinite-dimensional analogue of the Kempf-Ness functional. 

It is natural to try to apply the techniques used in the study of the K-energy functional to the complexified K-energy. For example, it is expected that the existence of critical points of~$\tilde{\m{M}}$, i.e.\ solutions of~\eqref{eq:scalar_dHYM_singola}, should be equivalent to the coerciveness of~$\tilde{\m{M}}$ on the space of almost calibrated representatives. To prove this however one would need a much better understanding of the geometry of this space; in general, it is not even clear when it is non-empty (which we assume in this paper), or connected.

A second issue is that a regularity theory for the geodesics appearing in Theorem~\ref{thm:complexified_Kenergy} is not yet available, and they do not encode a Riemannian structure on the space of almost calibrated potentials. We will discuss these issues in Section~\ref{sec:geod_equations}; we expect that it will be possible to combine the known regularity results for geodesics of K\"ahler potentials~\cite{Chen2000, TosattiWeinkove_C11regularity}, and geodesics of calibrated potentials~\cite{CollinsYau_dHYMgeodesics, CollinsChuLee_calibrated11forms} with respect to a fixed K\"ahler form, to establish the uniqueness of solutions of~\eqref{eq:scalar_dHYM_singola}, following the strategy of~\cite{BermanBerndtsson_uniquness}. To exemplify this, we prove a uniqueness result in Section~\ref{sec:uniqueness}.
\begin{thm}\label{thm:uniqueness}
	Assume that~$X$ is a complex surface and that~$\beta\sin\hat{\vartheta}>\alpha\cos\hat{\vartheta}$. Then, the almost calibrated solutions of~\eqref{eq:scalar_dHYM_singola} are unique, up to the action of the group of reduced automorphisms of~$X$. The same holds if~$(X,\alpha)$ is toric (of arbitrary dimension),~$\alpha^{\bb{C}}$ has supercritical phase, and one considers torus-invariant solutions of the equation.
\end{thm}
We refer to Section~\ref{sec:dHYM} for the definition of \emph{supercritical phase}, which can be regarded as a weak positivity property of the class~$\beta$, as is the condition~$\beta\sin\hat{\vartheta}>\alpha\cos\hat{\vartheta}$ in the case of surfaces. We remark that we do not expect the strategy used to establish Theorem~\ref{thm:uniqueness} to hold outside of the supercritical (or even hypercritical) range. This is motivated by a relation between the critical phase conditions on~$X$ and on a manifold of dimension $n+1$ that is useful in the study of the geodesic equation (see the discussion following Conjecture~\ref{conj:geodesics_existence}).

As a by-product of the construction used to prove Theorem \ref{thm:complexified_Kenergy}, we will also describe a Futaki-like invariant and a Matsushima criterion for equation \eqref{eq:scalar_dHYM_singola}.

\bigskip

\paragraph*{Acknowledgements.} I would like to thank Jacopo Stoppa for encouraging me to study this problem and for insightful remarks on a previous version of this paper. I also thank Ruadha\'i Dervan and Vestislav Apostolov for some interesting discussions and many useful comments. In particular, I thank Vestislav Apostolov for inspiring the argument in Remark~\ref{rmk:kernel}. I thank the anonymous reviewer for their insightful comments and suggestions.

\bigskip

\section{The scalar curvature equation with B-field}

In this section we introduce and recall some notions, particularly regarding the \emph{deformed Hermitian Yang-Mills equation}, that will be useful in studying equation \eqref{eq:scalar_dHYM_singola}.

First, notice that~\eqref{eq:scalar_dHYM_singola} is actually a system of two equations: as~$s(\omega)$ is a real function and~$c\in\bb{R}$, a necessary condition for~\eqref{eq:scalar_dHYM_singola} to hold is that also~$\gamma (B+\I\omega)^n/\omega^n$ is a real function. Integrating over~$X$, we see that~$\gamma$ must satisfy
\begin{equation}\label{eq:condition_classes}
	\gamma\frac{(\beta+\I\alpha)^n}{\alpha^n}\in\bb{R}.
\end{equation}
We assume throughout that $(\beta+\I\alpha)^n\not=0$. Then, \eqref{eq:condition_classes} uniquely determines $\mrm{e}^{\I\hat{\vartheta}}\coloneqq\gamma\abs{\gamma}^{-1}$. Note that $\hat{\vartheta}$ is not uniquely determined. Similarly, the constant~$c$ in~\eqref{eq:scalar_dHYM_singola} is uniquely determined by~$\alpha$,~$\beta$, and~$\gamma$, as
\begin{equation}
	\frac{n\,c_1(X).\alpha^{n-1}}{\alpha^n}=c+\gamma\frac{(\beta+\I\alpha)^n}{\alpha^n}.
\end{equation}
Hence, we write~\eqref{eq:scalar_dHYM_singola} as a system of equations
\begin{equation}\label{eq:scalar_dHYM}
\begin{dcases}
	\Im\left(\mrm{e}^{-\I\hat{\vartheta}}(B+\I\omega)^n\right)=0\\
	s(\omega)=c+\abs{\gamma}\frac{\Re\left(\mrm{e}^{-\I\hat{\vartheta}}(B+\I\omega)^n\right)}{\omega^n},
\end{dcases}
\end{equation}
which is the original form in which~\eqref{eq:scalar_dHYM_singola} was introduced, in~\cite{SchlitzerStoppa}. The first equation in~\eqref{eq:scalar_dHYM} is called the \emph{deformed Hermitian Yang-Mills} (dHYM) equation. It can be written as an equation for the eigenvalues of the endomorphism~$\omega^{-1}B\in\mrm{End}(T^{1,0}X)$, given by composing the identifications
\begin{equation}
	T^{1,0}X\xrightarrow{\quad B\quad } T^{0,1}{}^*X\xrightarrow{\quad \omega\quad }T^{1,0}X.
\end{equation}
If we denote by~$\lambda_1,\dots,\lambda_n$ the eigenvalues of this endomorphism, the dHYM equation is equivalent to
\begin{equation}\label{eq:dHYM_angles}
	\sum_{a=1}^n\arccot(\lambda_a)=\hat{\vartheta}\mod 2\pi,
\end{equation}
where (as in~\cite{GaoChen_Jeq_dHYM}) ``$\mrm{arccot}$'' is the inverse cotangent function with range~$(0,\pi)$, namely~$\arccot(x)=\frac{\pi}{2}-\arctan(x)$. The function~$\Theta(\omega^{\bb{C}})\coloneqq\sum\arccot(\lambda_a)$ is called the \emph{Lagrangian phase operator} of the complexified K\"ahler form. Similarly, if we define the \emph{Lagrangian radius operator} as~$r(\omega^{\bb{C}})\coloneqq\prod_a(1+\lambda_a^2)^{\frac{1}{2}}$, then the second equation in~\eqref{eq:scalar_dHYM} is equivalent to
\begin{equation}
	s(\omega)=c+\gamma\,r(\omega^{\bb{C}})\,\mrm{e}^{\I\Theta(\omega^{\bb{C}})}
\end{equation}
so that~\eqref{eq:scalar_dHYM} becomes
\begin{equation}
\begin{dcases}
	\Theta(\omega^{\bb{C}})+\mrm{arg}(\gamma)=0 \mod 2\pi\\
	s(\omega)=c+\abs{\gamma}\,r(\omega^{\bb{C}}).
\end{dcases}
\end{equation}

Jacob and Yau~\cite{JacobYau_special_Lag} proved that, for any \emph{fixed} K\"ahler form~$\omega\in\alpha$, a solution~$B$ of the dHYM equation
\begin{equation}
	\Im\left(\mrm{e}^{-\I\vartheta}(B+\I\omega)^n\right)=0
\end{equation}
is the unique minimizer of the volume functional
\begin{equation}
	V_\omega:B\mapsto\int_X\abs*{(B+\I\omega)^n}.
\end{equation}
Moreover, the value of~$V_\omega$ at a solution of the dHYM equation is~$\abs*{(\beta+\I\alpha)^n}$ and in particular depends only on the complexified K\"ahler class. We have a similar property of solutions of~\eqref{eq:scalar_dHYM}, see also~\cite[Lemma~$1.8$]{ScarpaStoppa_toric_dHYM}.
\begin{lemma}
	Solutions of the scalar curvature equation with B-field minimize the \emph{complexified Calabi functional}
	\begin{equation}
		(B,\omega)\mapsto\int_X\abs*{s(\omega)-\gamma\frac{(B+\I\omega)^n}{\omega^n}}^2\omega^n.
	\end{equation}
\end{lemma}
\begin{proof}
	Fix two forms~$\omega\in\alpha$ and~$B\in\beta$, and denote by~$\mrm{scB}_\gamma$ the complex-valued function~$s(\omega)-\gamma(B+\I\omega)^n/\omega^n$. Then the complexified Calabi functional is
	\begin{equation}
		(B,\omega)\mapsto\int_X\abs*{\mrm{scB}_\gamma}^2\omega^n
	\end{equation}
	and~$B+\I\omega$ solves the scalar curvature equation with~$B$-field if and only if~$\mrm{scB}_\gamma=c_\gamma$. The direct computation gives
	\begin{equation}
		\int_X\abs*{\mrm{scB}_\gamma}^2\omega^n=
		\int_X\abs*{\mrm{scB}_\gamma-c_\gamma}^2\omega^n
		+2\,c_\gamma\int_X\Re\left(\mrm{scB}_\gamma-c_\gamma\right)\omega^n
		+c_\gamma^2\,\mrm{Vol}(X,\omega)
	\end{equation}
	but the (real) constant~$c_\gamma$ is chosen precisely to satisfy
	\begin{equation}
		\int_X\Re\left(\mrm{scB}_\gamma-c_\gamma\right)\omega^n=0
	\end{equation}
	so that
	\begin{equation}
		\int_X\abs*{\mrm{scB}_\gamma}^2\omega^n=
		\int_X\abs*{\mrm{scB}_\gamma-c_\gamma}^2\omega^n
		+c_\gamma^2\,\mrm{Vol}(X,\omega)\geq c_\gamma^2\,\mrm{Vol}(X,\omega)
	\end{equation}
	and the equality is realised if and only if~$B+\I\omega$ solves~\eqref{eq:scalar_dHYM_singola}.
\end{proof}

\subsection{The deformed Hermitian Yang-Mills equation}\label{sec:dHYM}

For a fixed a K\"ahler form~$\omega\in\alpha$, the deformed Hermitian Yang-Mills equation was introduced in~\cite{LeungYauZaslow_dHYM} as a mirror to the special Lagrangian equation on Calabi-Yau manifolds. A systematic study of the properties of this equation on general K\"ahler manifolds started with~\cite{JacobYau_special_Lag}, and in recent years the dHYM equation has attracted great interest. The realisation that dHYM is a moment map equation~\cite{CollinsYau_dHYMgeodesics} suggested that there might be a stability notion characterising the existence of solutions~\cite{CollinsJacobYau_stability}, in the sense of Geometric Invariant Theory.

In this section, we give a brief account of some of the results about the dHYM equation that we need in our study of the system~\eqref{eq:scalar_dHYM}. The main notion we will need is that of \emph{supercritical phase} for the dHYM equation. For a fixed~$\omega\in\alpha$, the space of almost calibrated representatives of~$\beta$ is defined in~\cite{Collins_stabilitydHYM} as
\begin{equation}
	\m{H}_\omega=\set*{B\in\beta\tc\Re\left(\mrm{e}^{-\I\hat{\vartheta}}(B+\I\omega)^n\right)>0}
\end{equation}
and we can write it as using the Lagrangian phase operator as
\begin{equation}
	\m{H}_\omega=\bigsqcup_{\vartheta\in(0,n\pi),\ \mrm{e}^{\I\vartheta}=\mrm{e}^{\I\hat{\vartheta}}}
	\set*{B\in\beta\tc\left\lvert\Theta_\omega(B)-\vartheta\right\rvert<\frac{\pi}{2}}.
\end{equation}
In fact, at most one set of this disjoint union is non-empty, see~\cite[Lemma~$2.4$]{CollinsXieYau_dHYM}. It should be remarked that~$\m{H}_\omega$ might be empty.
\begin{definition}
	Assume that $\m{H}_\omega\not=\emptyset$. The \emph{lifted phase} of~$\beta$ with respect to~$\omega$, denoted $\hat{\vartheta}_\omega$, is the unique~$\hat{\vartheta}_\omega\in(0,n\pi)$ such that~$\mrm{e}^{\I\hat{\vartheta}_\omega}=\mrm{e}^{\I\hat{\vartheta}}$ and $$\set*{B\in\beta\tc\left\lvert\Theta_\omega(B)-\hat{\vartheta}_\omega\right\rvert<\frac{\pi}{2}}\not=\emptyset.$$
	We say that~$\beta$ has \emph{supercritical phase} if~$\hat{\vartheta}_\omega\in(0,\pi)$, and it has \emph{hypercritical phase} if~$\hat{\vartheta}_\omega\in(0,\pi/2)$.
\end{definition}
We will also need a similar notion of lifted phase for the complexified K\"ahler class~$\alpha^{\bb{C}}$; this might be problematic, as the lifted phase is, in general, not a topological quantity. We will therefore define below the supercritical phase condition only in the case when there are solutions of the dHYM equation in~$\alpha^{\bb{C}}$, see also~\cite[Proposition~$3.3$]{CollinsXieYau_dHYM} for related considerations.

The main conjecture that has guided research on the dHYM equation in recent years has been proposed in~\cite{CollinsJacobYau_stability}.
\begin{conj}\label{conj:dHYM_stability}
	Let~$(X,\omega)$ be a compact K\"ahler manifold, and let~$\beta$ be a~$(1,1)$-class with supercritical phase. Then there is a solution~$B\in\beta$ of the dHYM equation if and only if for all proper, irreducible analytic subvarieties~$V\subset X$ of dimension~$1\leq p< n$
	\begin{equation}
		\Im\left(\mrm{e}^{-\I\hat{\vartheta}}(\iota_V^*\alpha^{\bb{C}})^p\right)<0.
	\end{equation}
\end{conj}
This expectation was confirmed first for surfaces in~\cite{CollinsJacobYau_stability}, using results of~\cite{JacobYau_special_Lag}, and received partial confirmations and refinements. A rather general version of the conjecture was finally established in~\cite{GaoChen_Jeq_dHYM}, and later improved by~\cite{Takahashi_dHYM_MoishezonCriterion}, which is our main reference for the stability result in this section.
\begin{thm}[Corollary~$1.4$ and Corollary~$1.5$ in~\cite{Takahashi_dHYM_MoishezonCriterion}]\label{thm:dHYM_stability}
	Fix a complexified K\"ahler class~$\alpha^{\bb{C}}=\beta+\I\alpha$, and assume that~$\hat{\vartheta}\in(0,\pi)$ is such that~$\mrm{e}^{-\I\hat{\vartheta}}(\alpha^{\bb{C}})^n\in\bb{R}_{>0}$. Then, the following are equivalent:
	\begin{enumerate}
		\item there is a K\"ahler class~$\chi$ such that for every~$1\leq p\leq n$
		\begin{equation}
			\Im\left(\mrm{e}^{-\I\hat{\vartheta}}(\alpha^{\bb{C}})^p.\chi^{n-p}\right)\leq 0
		\end{equation}
		and for every analytic subvariety~$V\subset X$ of dimension~$1\leq m<n$ and every~$1\leq p\leq m$,
		\begin{equation}
			\Im\left(\mrm{e}^{-\I\hat{\vartheta}}(\iota_V^*\alpha^{\bb{C}})^p.\chi^{m-p}\right)<0;
		\end{equation}
		\item for any K\"ahler form~$\omega\in\alpha$, there exists a (unique) solution~$B\in\beta$ of the \emph{supercritical} dHYM equation (c.f. equation~\eqref{eq:dHYM_angles})
		\begin{equation}
			\sum\arccot(\lambda_a(\omega^{-1}B))=\hat{\vartheta}.
		\end{equation}
	\end{enumerate}
	In particular, if~$\alpha$ is the first Chern class of an ample line bundle~$L\to X$, then Conjecture~\ref{conj:dHYM_stability} holds.
\end{thm}
As the first condition in Theorem \ref{thm:dHYM_stability} only depends on $\alpha$ and $\beta$, we can use it to deduce some important properties of solutions of the dHYM equation for a varying K\"ahler form.
\begin{cor}\label{cor:smoothdependence}
	Assume that for some K\"ahler form $\omega_0\in\alpha$ there exists a solution of the dHYM equation~$B_0\in\beta$ that lies in the supercritical phase range, i.e.\
\[
	\Theta_{\omega_0}(B_0)\equiv\hat{\vartheta}\in(0,\pi).
\]
Then, for any~$\omega\in\alpha$ there is~$B_\omega\in\beta$ solving the dHYM equation, in the same phase range $(0,\pi)$. Moreover, the map $\omega\to B_\omega$ is continuous and differentiable.
\end{cor}
\begin{proof}
 	The first part is a direct consequence of Theorem \ref{thm:dHYM_stability}. As for the differentiability of $\omega\mapsto B_\omega$, this follows from the Implicit Function Theorem. The linearisation of the dHYM equation at $\omega$ is a uniformly elliptic second order operator $\Delta'$ \cite[\S$2$]{CollinsJacobYau_stability}, that in a system of coordinates where $\omega$ is the Euclidean metric and $B=\sum_j\I\lambda_jdz^j\wedge d\bar{z}^j$ has the form
 	\begin{equation}
 		\Delta'(\psi)=\textstyle\sum\nolimits_j(1+\lambda_j^2)\partial_j\partial_{\bar{j}}\psi.
	\end{equation}
	As this operator has trivial kernel (after normalising the B-field potentials $\psi$), we can use the IFT in the H\"older spaces $\m{C}^{k,\alpha}$ and the uniqueness of solutions to deduce that $\omega\mapsto B_\omega$ is smooth in any $\m{C}^k$ topology.
\end{proof}
Hence, it makes sense to pose the following.
\begin{definition}
	Let~$X$ be a compact complex manifold, and let~$\alpha^{\bb{C}}$ be a complexified K\"ahler class such that there exists an almost calibrated solution~$B+\I\omega\in\alpha^{\bb{C}}$ of the dHYM equation. We say that~$\alpha^{\bb{C}}$ has \emph{supercritical phase} if~$\Theta(B+\I\omega)\in(0,\pi)$. Likewise, if~$\Theta(B+\I\omega)\in(0,\pi/2)$ we say that~$\alpha^{\bb{C}}$ has \emph{hypercritical phase}.
\end{definition}
\begin{rmk}
	Conjecture~\ref{conj:dHYM_stability} was at least partially inspired by a stability notion for the J-equation, see \cite{LejmiSzekelyhidi_Jflow}.
\end{rmk}

\section{The complexified K-energy functional}\label{sec:complexified_Kenergy}

The goal of this section is to prove Theorem~\ref{thm:complexified_Kenergy}. We will describe the system~\eqref{eq:scalar_dHYM} as the Euler-Lagrange equation for a functional on the space~$\alpha^{\bb{C}}$. Equivalently (through the~$\partial\bar{\partial}$-Lemma), this can be considered as a functional on an open and convex subset of~$\m{C}^\infty(M,\mathbb{R})^2\cong\m{C}^\infty(M,\bb{C})$. It is defined, up to a constant, by its first variation:
\begin{equation}\label{eq:firstvariation}
\left(\mathrm{d}\tilde{\m{M}}\right)_{\omega^{\bb{C}}}(\varphi)=\sigma_{\omega^{\bb{C}}}(\varphi)\coloneqq \int_X\,\Im \left[\varphi\left(\gamma(\omega^{\bb{C}})^n-(s(\omega)-c_\gamma)\omega^n\right)\right].
\end{equation}
We call~$\tilde{\m{M}}$ the \emph{complexified K-energy} of the class~$\alpha^{\bb{C}}$. Here and in what follows we always let~$\omega^{\bb{C}}=B+\I\omega$.
 
We should check that the~$1$-form on the right-hand side of~\eqref{eq:firstvariation} is exact. It is sufficient to verify that it is closed; this amounts to the following claim.
\begin{lemma}
For every~$\varphi_1,\varphi_2\in\m{C}^\infty(M,\bb{C})$, the expression
\begin{equation}
\partial_{t=0}\left(\sigma_{\omega^{\bb{C}}+\I\del\delbar t\varphi_2}(\varphi_1)\right)
\end{equation}
is symmetric in~$1$ and~$2$.
\end{lemma}
\begin{proof}
For~$\gamma=0$, this is a standard result, used in the definition of the K-energy. So it will be enough to prove the result for~$\sigma'$ defined by
\begin{equation}
\sigma'_{\omega^{\bb{C}}}(\varphi)\coloneqq \int_X\,\Im\left[\varphi\,\gamma\,(\omega^{\bb{C}})^n\right].
\end{equation}
Note that~$\sigma'$ is essentially the \emph{complexified Calabi-Yau functional} considered in~\cite{CollinsYau_dHYMgeodesics}. The only difference is that we also must consider variations of the K\"ahler form~$\omega$ and we consider~$\omega^{\bb{C}}=B+\I\omega$ rather than~$\omega^{\bb{C}}=\omega+\I B$. A direct computation gives
\begin{equation}
\begin{split}
\partial_{t=0}\sigma'_{\omega^{\bb{C}}+\I\del\delbar t\varphi_2}(\varphi_1)&=n\int_X\,\Im \left[\gamma\,\varphi_1\,\I\del\delbar\varphi_2\wedge(\omega^{\bb{C}})^{n-1}\right]=\\
&=n\int_X\,\Im\left[\gamma\,\varphi_2\,\I\del\delbar\varphi_1\wedge(\omega^{\bb{C}})^{n-1}\right]
\end{split}
\end{equation}
where the last equality follows from integration by parts, noting that~$\I\del\delbar$ is a real operator.
\end{proof}
It is possible to find a more explicit expression for~$\tilde{\m{M}}$ from Chen-Tian's expansion of the K-energy and the expansion for the complexified Calabi functional in~\cite[Proposition~$2.14$]{CollinsYau_dHYMgeodesics}. Having fixed a reference point~$\omega^{\bb{C}}_0=B_0+\I\omega_0$ of the complexified K\"ahler class,~$\tilde{\m{M}}$ equals to
\begin{equation}\label{eq:Kenergy_decomposition}
\begin{split}
\tilde{\m{M}}(u+\I v)=&\m{H}(v)+\frac{c_\gamma}{n+1}\m{E}(v)-\m{E}_{\mrm{Ric}(\omega_0)}(v)+\frac{1}{n+1}\Im\left(\gamma\,\m{E}^{\bb{C}}(u+\I v)\right).
\end{split}
\end{equation}
Here~$\m{H}$ is the \emph{entropy functional} of the K\"ahler class~$\alpha$ (with respect to the reference form~$\omega_0$), that is
\begin{equation}
\m{H}(v)=\int\log\left(\frac{\omega_v^n}{\omega_0^n}\right)\omega_v^n,
\end{equation}
while~$\m{E}$,~$\m{E}_\eta$ are respectively the \emph{energy functional} of the K\"ahler class and its \emph{twisted} version, and~$\m{E}^{\bb{C}}$ is the \emph{complexified energy} of~$\alpha^{\bb{C}}$, defined respectively as
\begin{equation}
\begin{split}
\m{E}(v)\coloneqq &\int_Xv\sum_{j=0}^n\omega_0^j\wedge\omega_v^{n-j},\\
\m{E}_\eta(v)\coloneqq &\int_Xv\sum_{j=0}^{n-1}\eta\wedge\omega_0^j\wedge\omega_v^{n-j-1},\\
\m{E}^{\bb{C}}(\varphi)\coloneqq &\int_X\varphi\sum_{j=0}^n(\omega_0^{\bb{C}})^j\wedge(\omega_\varphi^{\bb{C}})^{n-j}.
\end{split}
\end{equation}

\subsection{Second variation of the complexified K-energy}

Fix a reference point~$\omega_0^{\bb{C}}$ in~$\alpha^{\bb{C}}$. Let~$\varphi_t\in\m{C}^\infty(X,\bb{C})$ be a smooth path of complexified potentials. We can compute the second variation of~$\tilde{\m{M}}$ along~$\omega_t^\bb{C}=\omega_0^{\bb{C}}+\I\del\delbar\varphi_t$, similarly to what is done for the K-energy in the cscK problem.
\begin{equation}\label{eq:secondvar}
\begin{split}
\partial^2_t\tilde{\m{M}}(\varphi_t)=&\partial_t\sigma_{\omega_t^{\bb{C}}}(\dot{\varphi}_t)=\partial_t\int_X\,\Im\left[\dot{\varphi}_t\big(\gamma(\omega_t^{\bb{C}})^n-(s(\omega_t)-c_\gamma)\omega_t^n\big)\right]=\\
=&\int_X\,\Im\left[\gamma\left(\ddot{\varphi}_t(\omega_t^{\bb{C}})^n-n\,\I\del\dot{\varphi}_t\wedge\delbar\dot{\varphi}_t\wedge(\omega_t^{\bb{C}})^{n-1}\right)\right]+\\&+\norm*{\m{D}_{\omega_t}\,\Im(\dot{\varphi_t})}^2_{L^2(\omega_t)}-\int_X\left(\Im(\ddot{\varphi}_t)-\norm*{\del\Im(\dot{\varphi}_t)}^2_{\omega_t}\right)(s(\omega_t)-c_\gamma)\omega_t^n.
\end{split}
\end{equation}
It is useful to have a more explicit expression for this second variation. To ease notation, we stop indicating the dependence on~$t$ explicitly; also, decompose~$\varphi$ in real and imaginary parts as~$\varphi=\varphi_1+\I\varphi_2$. Then
\begin{equation}\label{eq:secondvar_expanded}
\begin{split}
\partial^2_t\tilde{\m{M}}(\varphi_t)=&\norm*{\m{D}\,\dot{\varphi_2}}^2_{L^2(\omega)}-\int_X\left(\ddot{\varphi}_2-\norm*{\del\dot{\varphi}_2}_\omega^2\right)\left(\left(s(\omega)-c_\gamma\right)\omega^n-\Re\left(\gamma(\omega^{\bb{C}})^n\right)\right)\\
&+\int_X\ddot{\varphi}_1\Im\left(\gamma(\omega^{\bb{C}})^n\right)+\int_X\norm*{\del\dot{\varphi}_2}^2\Re \left(\gamma(\omega^{\bb{C}})^n\right)\\
&-\int_X\left(\I\del\dot{\varphi}_1\wedge\delbar\dot{\varphi}_1-\I\del\dot{\varphi}_2\wedge\delbar\dot{\varphi}_2\right)\wedge\Im\left[\gamma(\omega^{\bb{C}})^{n-1}\right]\\
&-n\int_X\left(\I\del\dot{\varphi}_2\wedge\delbar\dot{\varphi}_1+\I\del\dot{\varphi}_1\wedge\delbar\dot{\varphi}_2\right)\wedge\Re\left[\gamma(\omega^{\bb{C}})^{n-1}\right].
\end{split}
\end{equation}
From this expression, it is relatively easy to deduce
\begin{lemma}\label{lemma:secondvar}
	The Hessian of~$\tilde{\m{M}}$ at an almost calibrated critical point is positive semi definite.
\end{lemma}
\begin{proof}
A critical point~$\varphi$ of~$\tilde{\m{M}}$ gives a solution~$\omega^{\bb{C}}\coloneqq \omega^{\bb{C}}_\varphi$ of~\eqref{eq:scalar_dHYM}. So, consider the second variation of~$\tilde{\m{M}}$ along a path~$\varphi_t$ such that~$\varphi_0=\varphi$. Equation~\eqref{eq:secondvar_expanded} gives
\begin{equation}\label{eq:secondvar_critpoint}
\begin{split}
\partial^2_t\tilde{\m{M}}(\varphi_t)=&\norm*{\m{D}\,\dot{\varphi_2}}^2_{L^2(\omega)}+\int_X\norm*{\del\dot{\varphi}_2}^2\Re\left(\gamma(\omega^{\bb{C}})^n\right)\\
&-n\int_X\left(\I\del\dot{\varphi}_1\wedge\delbar\dot{\varphi}_1-\I\del\dot{\varphi}_2\wedge\delbar\dot{\varphi}_2\right)\wedge\Im\left[\gamma(\omega^{\bb{C}})^{n-1}\right]\\
&-n\int_X\left(\I\del\dot{\varphi}_2\wedge\delbar\dot{\varphi}_1+\I\del\dot{\varphi}_1\wedge\delbar\dot{\varphi}_2\right)\wedge\Re\left[\gamma(\omega^{\bb{C}})^{n-1}\right].
\end{split}
\end{equation}
The first term on the right-hand side of~\eqref{eq:secondvar_critpoint} is non-negative. We now show that the sum of the other terms is non-negative, by computing the sum of the integrands.

Consider a frame~$\zeta^1,\dots,\zeta^n$ for~$T^*X$ in which~$\omega$ is the standard symplectic form~$\omega=\I\sum_a\zeta^a\wedge\bar{\zeta}^a$ and~$B$ is diagonal,~$B=\I\sum_a\lambda_a\zeta^a\wedge\bar{\zeta}^a$. We can write~$\del\dot{\varphi}_1=u_a\zeta^a$ and~$\del\dot{\varphi}_2=v_a\zeta^a$ for some locally defined function~$u_1,\dots,u_n$,~$v_1,\dots,v_n\in\m{C}^\infty(X,\bb{C})$. We can then define the radii and angles as~$r_a\coloneqq(1+\lambda^2_a)^{\frac{1}{2}}$ and~$\vartheta_a\coloneqq\arccot(\lambda_a)$.

Let~$\gamma=\abs{\gamma}\,\mrm{e}^{-\I\hat{\vartheta}}$, and recall that~$\hat{\vartheta}$ is determined uniquely by the topology, up to multiples of~$2\pi$. As~$\omega$ and~$B$ solve the dHYM equation,~$\mrm{exp}(-\I\hat{\vartheta}+\I\sum\vartheta_a)=1$, and the integrands on the right-hand side of~\eqref{eq:secondvar_critpoint} become
\begin{equation}
\begin{gathered}
\norm*{\del\dot{\varphi}_2}^2\Re \left(\gamma(\omega^{\bb{C}})^n\right)=\abs{\gamma}\,\omega^n\,r(\omega^{\bb{C}})\,\sum\abs{v_a}^2\\
n\left(\I\del\dot{\varphi}_1\wedge\delbar\dot{\varphi}_1-\I\del\dot{\varphi}_2\wedge\delbar\dot{\varphi}_2\right)\wedge\Im\left[\gamma(\omega^{\bb{C}})^{n-1}\right]=-\abs{\gamma}\,\omega^n\,r(\omega^{\bb{C}})\,\sum\frac{\abs{u_a}^2-\abs{v_a}^2}{r_a}\sin(\vartheta_a)\\
n\left(\I\del\dot{\varphi}_2\wedge\delbar\dot{\varphi}_1+\I\del\dot{\varphi}_1\wedge\delbar\dot{\varphi}_2\right)\wedge\Re\left[\gamma(\omega^{\bb{C}})^{n-1}\right]=\abs{\gamma}\,\omega^n\,r(\omega^{\bb{C}})\,\sum\frac{u_a\bar{v}_a+\bar{u}_av_a}{r_a}\cos(\vartheta_a).
\end{gathered}
\end{equation}
To prove that the sum of the integrands is positive, it will be enough to show that, for each~$a\in\set{1,\dots,n}$
\begin{equation}\label{eq:autoval}
\abs{v_a}^2+\frac{\abs{u_a}^2-\abs{v_a}^2}{r_a}\sin(\vartheta_a)-\frac{u_a\bar{v}_a+\bar{u}_av_a}{r_a}\cos(\vartheta_a)
\end{equation}
is non-negative. As~$\vartheta_a=\arccot\lambda_a=\frac{\pi}{2}-\arctan\lambda_a$, we can rewrite~\eqref{eq:autoval} as
\begin{equation}
\abs{v_a}^2+\frac{\abs{u_a}^2-\abs{v_a}^2}{1+\lambda_a^2}-\lambda_a\frac{u_a\bar{v}_a+\bar{u}_av_a}{1+\lambda_a^2}=\frac{1}{1+\lambda_a^2}\left(\lambda_a^2\abs{v_a}^2+\abs{u_a}^2-\lambda_a(u_a\bar{v}_a+\bar{u}_av_a)\right),
\end{equation}
but this is just~$(1+\lambda_a^2)^{-1}\abs*{\lambda_av_a-u_a}^2$.
\end{proof}
\begin{rmk}\label{rmk:convex_along_dHYM}
From the proof of Lemma~\ref{lemma:secondvar}, it is also clear that if~$\varphi_1(t),\varphi_2(t)\in\m{C}^\infty(X,\bb{R})$ are two families of functions such that~$\ddot{\varphi}_2=\norm{\dot{\varphi}_2}^2_{\omega_{\varphi_2}}$ (i.e.~$\varphi_2$ is a K\"ahler geodesic) and~$B+\I\del\delbar\varphi_1$ is a solution of the dHYM equation with respect to~$\omega_{\varphi_2}$, then~$\tilde{\m{M}}$ is convex along the path~$\varphi_1(t)+\I\varphi_2(t)$.
\end{rmk}
\begin{rmk}\label{rmk:kernel}
The proof of Lemma~\ref{lemma:secondvar} tells us something more: the kernel of the linearization of~\eqref{eq:scalar_dHYM} around a solution~$(\omega,B)$ is given by pairs~$\varphi_1,\varphi_2\in\m{C}^\infty(X,\bb{R})$ such that
\begin{align}
&\delbar\left(\nabla^{1,0}\varphi_2\right)=0\label{eq:holpotential_omega}\\
&\nabla^{1,0}\varphi_2\lrcorner B=\I\delbar\varphi_1.\label{eq:holpotential_beta}
\end{align}
We claim that for each~$\varphi_2\in\m{C}^\infty(X,\bb{R})$ that satisfies~\eqref{eq:holpotential_omega}, 
there exists a~$\varphi_1$, unique up to the addition of a constant, satisfying~\eqref{eq:holpotential_beta}. Indeed,~\eqref{eq:holpotential_omega} tells us that~$Z\coloneqq\nabla^{1,0}\varphi_2$ is a holomorphic vector field with a holomorphy potential with respect to~$\omega$. So for any other K\"ahler metric~$\omega'$, there is some~$\varphi'$ such that~$Z=\nabla^{1,0}_{\omega'}\varphi'$, see~\cite{LeBrun-Simanca}. In particular, this holds for the K\"ahler metric~$\omega+t\,B$, for some small~$t$. In other words, there is some function~$\psi$ such that
\begin{equation}
\delbar\psi=Z\lrcorner(\omega+t\,B)=\delbar\varphi_2+t\,Z\lrcorner B
\end{equation}
so~$\varphi_1=t^{-1}(\psi-\varphi_2)$ will satisfy~\eqref{eq:holpotential_beta}. It is not guaranteed, however, that we can choose~$\varphi_1$ to be a \emph{real} function: this happens if and only if~$\m{L}_{J\nabla\varphi_2}B=0$. To see this, notice that~\eqref{eq:holpotential_beta} implies
\begin{equation}
\m{L}_{J\nabla\varphi_2}B=\mrm{d}\left((J\nabla\varphi_1)\lrcorner B\right)=\del\delbar(\bar{\varphi}_1-\varphi_1)
\end{equation}
so~$\m{L}_{J\nabla\varphi_2}B=0$ if and only if~$\varphi_1$ is real (up to a constant). Now, recall that we are assuming~$(\omega,B)$ to be a solution of the dHYM equation; if~$f_t$ is the flow of the vector field~$J\nabla\varphi_2$, also~$(f_t^*\omega,f_t^*B)$ will be a solution of dHYM. But~$f_t$ preserves~$\omega$, as it is the flow of a Hamiltonian vector field, and since for any fixed~$\omega$ there is a unique solution of dHYM we conclude that also~$f_t^*B=B$ for every~$t$. In particular,~$\m{L}_{J\nabla\varphi_2}B=0$.

Summing up, we have seen that the kernel of the linearisation of~\eqref{eq:scalar_dHYM} around a solution~$(\omega,B)$ is in one-to-one correspondence with the set of holomorphy potentials
\begin{equation}\label{eq:kernel_liner}
\set*{\varphi\in\m{C}^\infty(X,\bb{R})\tc\delbar\nabla^{1,0}_\omega\varphi=0}/\bb{R}.
\end{equation}
\end{rmk}
As these potentials generate holomorphic vector fields on $X$, we find
\begin{cor}
	Let $\omega^{\bb{C}}$ be an almost calibrated critical point of $\tilde{\m{M}}$, and suppose that $\mrm{Aut}(X)$ is discrete. Then, $\omega^{\bb{C}}$ is a local minimum of $\tilde{\m{M}}$.
\end{cor}

The scalar curvature equation with~$B$-field is a moment map equation for an infinite-dimensional Hamiltonian action; when~$\beta$ is a Hodge class, this was shown in~\cite{SchlitzerStoppa}, while the result for general classes is due to~\cite{ScarpaStoppa_toric_dHYM}. Having identified the kernel of the linearisation around a solution as~\eqref{eq:kernel_liner}, we obtain from the general theory of~\cite{Wang_Futaki} an analogue of the classical obstructions to the existence of cscK metrics, namely Matsushima's criterion and the vanishing of the Futaki invariant.

A Futaki-like invariant was already introduced  in~\cite{ScarpaStoppa_toric_dHYM}; essentially, this coincides with the~$1$-form~\eqref{eq:firstvariation} evaluated on the complexification of~\eqref{eq:kernel_liner}
\begin{equation}
\f{h}_0\coloneqq\set*{\xi\in H^0(X,T^{1,0}X)\tc\xi\mbox{ has a zero}}\cong\set*{\varphi\in\m{C}^\infty(X,\bb{C})\tc\delbar\nabla^{1,0}_\omega\varphi=0}/\bb{C}.
\end{equation}
More precisely, for every~$\xi\in\f{h}_0$ and any~$2$-form~$\eta$ there is a function~$\varphi(\xi,\eta)$ such that~$\xi\lrcorner\eta=\I\delbar\varphi(\xi,\eta)$. Hence we can define a map~$\f{h}_0\to\bb{C}$ by
\begin{equation}
\m{F}_{\omega,B}(\xi)=\int\varphi(\xi,\omega)\big(c_\gamma+\Re\left(\gamma(B+\I\omega)^n\right)-s(\omega)\omega^n\big)+\int\varphi(\xi,B)\,\Im\left(\gamma(B+\I\omega)^n\right).
\end{equation}
It does not depend on the choice of~$B+\I\omega\in\alpha^{\bb{C}}$ and vanishes if there is a solution of~\eqref{eq:scalar_dHYM_singola}.

In Remark~\ref{rmk:kernel} we noticed that if~$B+\I\omega$ solves~\eqref{eq:scalar_dHYM_singola}, then for any~$\xi\in\f{h}_0$, the potential~$\varphi(\xi,\omega)$ is a real function if and only if~$\varphi(\xi,B)$ is. These vector fields are all Killing, so from~\cite[Corollary~$12$]{Wang_Futaki} we obtain
\begin{cor}
If there is a solution of~\eqref{eq:scalar_dHYM_singola}, then~$\f{h}_0$ is a reductive Lie algebra.
\end{cor}

\subsection{Convexity and geodesics}

We can check that~$\tilde{\m{M}}$ is also convex along other paths, not necessarily satisfying the dHYM equation. First, we define a class of representatives for the complexified K\"ahler class (and a set of potentials), following~\cite{CollinsYau_dHYMgeodesics}.
\begin{definition}
	A complexified K\"ahler form~$\omega^{\bb{C}}\in\alpha^{\bb{C}}$ is an \emph{almost calibrated representative} of~$\alpha^{\bb{C}}$ if it satisfies~$\Re(\mrm{e}^{-\I\hat{\vartheta}}(\omega^{\bb{C}})^n)>0$.
\end{definition}
Fixing some~$\omega^{\bb{C}}_0\in\alpha^{\bb{C}}$, the almost calibrated representatives are parametrized, through the~$\del\delbar$-Lemma, by the space of \emph{almost calibrated potentials}
\begin{equation}
\m{H}=\set*{\varphi\in\m{C}^\infty(X,\bb{C})\tc \Im(\omega^{\bb{C}}_0+\I\del\delbar\varphi) \text{ is K\"ahler, and }\Re\left(\mrm{e}^{-\I\hat{\vartheta}}(\omega_0^{\bb{C}}+\I\del\delbar\varphi)^n\right)>0}.
\end{equation}
The existence of almost calibrated representatives of~$\alpha^{\bb{C}}$ is a necessary condition for the existence of solutions of the dHYM equation: we henceforth assume that~$\m{H}\not=\emptyset$.

Notice also that, if a lift of~$\mrm{e}^{\I\hat{\vartheta}}$ has been fixed (e.g.\ if~$\alpha^{\bb{C}}$ is supercritical), then the space of almost calibrated representatives can be characterized in terms of the Lagrangian phase operator as
\begin{equation}
\set*{\omega^{\bb{C}}\in\alpha^{\bb{C}}\tc\hat{\vartheta}-\frac{\pi}{2}<\Theta(\omega^{\bb{C}})<\hat{\vartheta}+\frac{\pi}{2}}.
\end{equation}
\begin{prop}\label{prop:convex_geod}
Assume that~$\varphi_t\in\m{H}$ is a path of potentials satisfying
\begin{equation}\label{eq:geodesic_coupled}
\begin{dcases}
&\Im(\ddot{\varphi})-\norm{\del\Im(\dot{\varphi})}^2_{\omega_t}=0\\
&\Re\left[\mrm{e}^{-\I\hat{\vartheta}}\left(\ddot{\varphi}\,(\omega^{\bb{C}}_{\varphi_t})^n-n\,\I\del\dot{\varphi}\wedge\delbar\dot{\varphi}\wedge(\omega^{\bb{C}}_{\varphi_t})^{n-1}\right)\right]=0.
\end{dcases}
\end{equation}
Then~$\tilde{\m{M}}$ is convex along~$\varphi_t$.
\end{prop}
\begin{proof}
From equation~\eqref{eq:secondvar}, if~$\Im(\ddot{\varphi})-\norm{\del\Im(\dot{\varphi})}^2_{\omega_t}=0$ we have
\begin{equation}\label{eq:secondvar_geod}
\partial^2_t\m{M}(\varphi_t)=\norm*{\m{D}\,\Im(\dot{\varphi})}^2_{L^2(\omega_{\Im(\varphi)})}+\int_X\,\Im\left[\gamma\left(\ddot{\varphi}\,(\omega^{\bb{C}}_\varphi)^n-n\,\I\del\dot{\varphi}\wedge\delbar\dot{\varphi}\wedge(\omega^{\bb{C}}_\varphi)^{n-1}\right)\right].
\end{equation}
We check that, under our hypothesis, the second term on the right-hand side of~\eqref{eq:secondvar_geod} is non-negative. To do so, we use the same local coframe and the same notation as in the proof of Lemma~\ref{lemma:secondvar}. In addition to that, let also~$\eta=\Theta-\hat{\vartheta}$. Then, we have
\begin{equation}\label{eq:complex_forms}
\begin{split}
&\gamma\,\ddot{\varphi}\,(\omega^{\bb{C}}_\varphi)^n=\omega^n\abs{\gamma}r(\omega^{\bb{C}})\mrm{e}^{\I\eta}\ddot{\varphi}\\
&\gamma\,n\,\I\del\dot{\varphi}\wedge\delbar\dot{\varphi}\wedge(\omega^{\bb{C}}_\varphi)^{n-1}=\omega^n\abs{\gamma}r(\omega^{\bb{C}})\mrm{e}^{\I\eta}\sum_{a=1}^n(\del\dot{\varphi})_a(\delbar\dot{\varphi})_ar_a^{-1}\mrm{e}^{-\I\vartheta_a}.
\end{split}
\end{equation}
Let~$\varphi=u+\I v$ for~$u,v\in\m{C}^\infty(X,\bb{R})$. Then we can expand~\eqref{eq:complex_forms} into real and imaginary parts as
\begin{equation}
\begin{split}
\Re\left(\gamma\,\ddot{\varphi}\,(\omega^{\bb{C}}_\varphi)^n\right)=\abs{\gamma}r(\omega^{\bb{C}})\left(\ddot{u}\cos\eta-\ddot{v}\sin\eta\right)\omega^n\\
\Im\left(\gamma\,\ddot{\varphi}\,(\omega^{\bb{C}}_\varphi)^n\right)=\abs{\gamma}r(\omega^{\bb{C}})\left(\ddot{u}\sin\eta+\ddot{v}\cos\eta\right)\omega^n
\end{split}
\end{equation}
\begin{equation}
\begin{split}
\Re\big(\gamma\,n\,\I\del\dot{\varphi}\wedge\delbar\dot{\varphi}\wedge&(\omega^{\bb{C}}_\varphi)^{n-1}\big)=\\
=&\abs{\gamma}\omega^nr(\omega^{\bb{C}})\sum_{a=1}^n\frac{\abs{(\del\dot{u})_a}^2-\abs{(\del\dot{v})_a}^2}{1+\lambda_a^2}\left(\lambda_a\cos\eta+\sin\eta\right)\\
&-\abs{\gamma}\omega^nr(\omega^{\bb{C}})\sum_{a=1}^n\frac{(\del\dot{u})_a(\delbar\dot{v})_a+(\del\dot{v})_a(\delbar\dot{u})_a}{1+\lambda_a^2}\left(\lambda_a\sin\eta-\cos\eta\right)
\end{split}
\end{equation}
\begin{equation}
\begin{split}
\Im\big(\gamma\,n\,\I\del\dot{\varphi}\wedge\delbar\dot{\varphi}\wedge&(\omega^{\bb{C}}_\varphi)^{n-1}\big)=\\
=&\abs{\gamma}\omega^nr(\omega^{\bb{C}})\sum_{a=1}^n\frac{\abs{(\del\dot{u})_a}^2-\abs{(\del\dot{v})_a}^2}{1+\lambda_a^2}\left(\lambda_a\sin\eta-\cos\eta\right)\\
&+\abs{\gamma}\omega^nr(\omega^{\bb{C}})\sum_{a=1}^n\frac{(\del\dot{u}_1)_a(\delbar\dot{v}_2)_a+(\del\dot{\varphi}_2)_a(\delbar\dot{\varphi}_1)_a}{1+\lambda_a^2}\left(\lambda_a\cos\eta+\sin\eta\right).
\end{split}
\end{equation}
The second condition in~\eqref{eq:geodesic_coupled} then becomes
\begin{equation}
\begin{split}
\ddot{u}\cos\eta-\ddot{v}\sin\eta=&\sum_{a=1}^n\frac{\abs{(\del\dot{u})_a}^2-\abs{(\del\dot{v})_a}^2}{1+\lambda_a^2}\left(\lambda_a\cos\eta+\sin\eta\right)\\
&-\sum_{a=1}^n\frac{(\del\dot{u})_a(\delbar\dot{v})_a+(\del\dot{v})_a(\delbar\dot{u})_a}{1+\lambda_a^2}\left(\lambda_a\sin\eta-\cos\eta\right)
\end{split}
\end{equation}
so using that~$v$ is a K\"ahler geodesic, we see that~$u$ must satisfy
\begin{equation}\label{eq:ddot_Re_phi}
\begin{split}
\ddot{u}\cos\eta=&\sum\abs{(\del\dot{v})_a}^2\sin\eta+\sum_{a=1}^n\frac{\abs{(\del\dot{u})_a}^2-\abs{(\del\dot{v})_a}^2}{1+\lambda_a^2}\left(\lambda_a\cos\eta+\sin\eta\right)\\
&-\sum_{a=1}^n\frac{(\del\dot{u})_a(\delbar\dot{v})_a+(\del\dot{v})_a(\delbar\dot{u})_a}{1+\lambda_a^2}\left(\lambda_a\sin\eta-\cos\eta\right).
\end{split}
\end{equation}
On the other hand, the second integrand in~\eqref{eq:secondvar_geod} (up to factoring out~$\abs{\gamma}r(\omega^{\bb{C}})\omega^n$) is
\begin{equation}\label{eq:integrand}
\begin{split}
\Im\big[\mrm{e}^{-\I\hat{\vartheta}}&\left(\ddot{\varphi}\,(\omega^{\bb{C}}_\varphi)^n-n\,\I\del\dot{\varphi}\wedge\delbar\dot{\varphi}\wedge(\omega^{\bb{C}}_\varphi)^{n-1}\right)\big]/\omega^nr(\omega^{\bb{C}})=\\
=&\ddot{u}\sin\eta+\sum\abs{(\del\dot{v})_a}^2\cos\eta\\
&-\sum_{a=1}^n\frac{\abs{(\del\dot{u})_a}^2-\abs{(\del\dot{v})_a}^2}{1+\lambda_a^2}\left(\lambda_a\sin\eta-\cos\eta\right)\\
&-\sum_{a=1}^n\frac{(\del\dot{u})_a(\delbar\dot{v})_a+(\del\dot{v})_a(\delbar\dot{u})_a}{1+\lambda_a^2}\left(\lambda_a\cos\eta+\sin\eta\right).
\end{split}
\end{equation}
Plug in~\eqref{eq:integrand} the expression for~$\ddot{v}$ obtained from~\eqref{eq:ddot_Re_phi} to obtain
\begin{equation}\label{eq:imaginpart_geod}
\begin{split}
\Im\big[\gamma\left(\ddot{\varphi}\,(\omega^{\bb{C}}_\varphi)^n-n\,\I\del\dot{\varphi}\wedge\delbar\dot{\varphi}\wedge(\omega^{\bb{C}}_\varphi)^{n-1}\right)\big]/\omega^nr(\omega^{\bb{C}})=
\phantom{\left(+\frac{\sin^2\eta}{\cos\eta}\right)}&\\
=\frac{\abs{\gamma}}{\cos\eta}\sum_{a=1}^{n}\Bigg[\abs{(\del\dot{v})_a}^2\sin^2\eta
+\frac{\abs{(\del\dot{u})_a}^2-\abs{(\del\dot{v})_a}^2}{1+\lambda_a^2}\left(\lambda_a\cos\eta\sin\eta+\sin^2\eta\right)&\\
-\frac{(\del\dot{u})_a(\delbar\dot{v})_a+(\del\dot{v})_a(\delbar\dot{u})_a}{1+\lambda_a^2}\left(\lambda_a\sin^2\eta-\cos\eta\sin\eta\right)&\\
+\abs{(\del\dot{v})_a}^2\cos^2\eta-\frac{\abs{(\del\dot{u})_a}^2-\abs{(\del\dot{v})_a}^2}{1+\lambda_a^2}\left(\lambda_a\cos\eta\sin\eta-\cos^2\eta\right)\\
-\frac{(\del\dot{u})_a(\delbar\dot{v})_a+(\del\dot{v})_a(\delbar\dot{u})_a}{1+\lambda_a^2}\left(\lambda_a\cos^2\eta+\cos\eta\sin\eta\right)&\Bigg].
\end{split}
\end{equation}
A few simplifications allow to rewrite the summand, for a fixed index~$a$, as
\begin{equation}\label{eq:realpart_summand}
\abs{(\del\dot{v})_a}^2+\frac{\abs{(\del\dot{u})_a}^2-\abs{(\del\dot{v})_a}^2}{1+\lambda_a^2}-\lambda_a\frac{(\del\dot{u})_a\,(\delbar\dot{v})_a+(\del\dot{v})_a\,(\delbar\dot{u})_a}{1+\lambda_a^2}=\frac{\abs{(\del\dot{u})_a-\lambda_a\,(\del\dot{v})_a}^2}{1+\lambda_a^2}.
\end{equation}
As~$\varphi\in\m{H}$ if and only if~$\cos\eta>0$, this concludes the proof.
\end{proof}

\subsubsection{The complexified geodesic equation}\label{sec:geod_equations}

Even though we are calling~\eqref{eq:geodesic_coupled} the ``geodesic equations'' on the space~$\m{H}$ of almost calibrated representatives, they do not describe the geodesics for a metric on this infinite-dimensional space. A similar phenomenon was first noticed in~\cite{FernandezPradaConsul_coupledKahlerHYM}, where a similar set of equations was proposed to study the \emph{coupled K\"ahler-Yang-Mills system}. Equation~\eqref{eq:scalar_dHYM} shares formally many properties with this coupled system, and in fact there is a limiting regime in which~\eqref{eq:scalar_dHYM} reduces to the equations considered in~\cite{FernandezPradaConsul_coupledKahlerHYM} for a line bundle (see~\cite{SchlitzerStoppa} for more details).

Following the argument in~\cite[Proposition~$3.17$]{FernandezPradaConsul_coupledKahlerHYM}, it can be seen that, rather than being the geodesic equation for a Levi-Civita connection on~$\m{H}$,~\eqref{eq:geodesic_coupled} describes the geodesics for a torsion-free connection on~$\m{H}$, seen as a bundle over the K\"ahler class~$\alpha$ whose fibres are the space of almost-calibrated representatives of~\cite{CollinsYau_dHYMgeodesics}.

We now describe a set of solutions to~\eqref{eq:geodesic_coupled} coming from elements of~$\mrm{Aut}(X,\alpha^{\bb{C}})$, which we call \emph{trivial geodesics}. Fix a reference point~$\omega^{\bb{C}}_0=B_0+\I\omega_0$ in~$\alpha^{\bb{C}}$ and assume that~$(u,v)$ are real functions on~$X$ that solve~\eqref{eq:holpotential_omega} and~\eqref{eq:holpotential_beta} with respect to~$\omega_0$ and~$B_0$. Note in particular that~$\nabla_0v$ is a real-holomorphic vector field on~$X$. We let~$f_t$ be the flow of~$\frac{1}{2}\nabla_0v$, and define~$\psi_t$ by~$\psi_t=\int_0^tf_s^*(u+\I v)\mrm{d}s$. Then, direct computation shows
\begin{equation}
\omega_0^{\bb{C}}+\I\del\delbar\psi_t=f_t^*\omega_0^{\bb{C}},
\end{equation}
and it follows from~\eqref{eq:holpotential_beta} that
\begin{equation}\label{eq:trivial_geod}
\nabla^{1,0}_{t}\Im(\dot{\psi}_t)\lrcorner B_t=\I\delbar\Re(\dot{\psi}_t).
\end{equation}
We claim that~$\psi_t$ solves the geodesic system~\eqref{eq:geodesic_coupled}.
\begin{proof}
We let~$\omega_t\coloneqq f_t^*\omega_0=\omega_0+\I\del\delbar\Re(\psi_t)$. Then, from the definition of~$\psi_t$ we get
\begin{equation}\label{eq:hol_geod_dd}
\begin{gathered}
\ddot{\psi}=\partial_t(f_t^*(u+\I v))=\frac{1}{2}f_t^*(\m{L}_{\nabla_0\varphi_2}(u+\I v))=\frac{1}{2}\left(\nabla_t\Im(\dot{\psi})\right)(\dot{\psi})=\\
=\frac{1}{2}\left(
\left\langle\mrm{d}\Im(\dot{\psi}),\mrm{d}\Re(\dot{\psi})\right\rangle_t
+\I\left\langle\mrm{d}\Im(\dot{\psi}),\mrm{d}\Im(\dot{\psi})\right\rangle_t
\right)
\end{gathered}
\end{equation}
from which it is clear that the first equation in~\eqref{eq:geodesic_coupled} is satisfied. To show that also the second equation holds, we fix a coframe~$\zeta^1,\dots,\zeta^n$ adapted to the complexified K\"ahler form~$\omega^{\bb{C}}_t$ as in the proof of Lemma~\ref{lemma:secondvar}, and use the same notation. We also denote by~$u_1,\dots,u_n$ and~$v_1,\dots,v_n$ the components of~$\del\Re(\dot{\psi})$ and~$\del\Im(\dot{\psi})$, respectively.

Taking the real part of~\eqref{eq:hol_geod_dd} then gives us
\begin{equation}\label{eq:hol_geod_dd_imaginary}
\Re(\ddot{\psi})=\frac{1}{2}\sum_{a=1}^nu_a\bar{v}_a+v_a\bar{u}_a.
\end{equation}
Similarly, we can rephrase~$\nabla^{1,0}_{t}\Im(\dot{\psi}_t)\lrcorner B_t=\I\delbar\Re(\dot{\psi}_t)$ (after conjugation) as
\begin{equation}\label{eq:kernel_coords}
\lambda_av_a=u_a.
\end{equation}
Putting together~\eqref{eq:ddot_Re_phi} and~\eqref{eq:hol_geod_dd_imaginary}, we see that the second equation in~\eqref{eq:geodesic_coupled} is equivalent to
\begin{equation}
\begin{split}
&\frac{1}{2}\sum(u_a\bar{v}_a+v_a\bar{u}_a)\cos\eta+\sum\frac{u_a\bar{v}_a+v_a\bar{u}_a}{1+\lambda_a^2}\left(\lambda_a\sin\eta-\cos\eta\right)=\\
&=\sum\abs{v_a}^2\sin\eta+\sum\frac{\abs{u_a}^2-\abs{v_a}^2}{1+\lambda_a^2}\left(\lambda_a\cos\eta+\sin\eta\right).
\end{split}
\end{equation}
Taking into account~\eqref{eq:kernel_coords}, this becomes
\begin{equation}
\begin{split}
&\cos\eta\sum\abs{v_a}^2\left(\lambda_a
-\frac{2\lambda_a}{1+\lambda_a^2}
-\frac{\lambda_a^2-1}{1+\lambda_a^2}\lambda_a\right)=\\
&=\sin\eta\sum\abs{v_a}^2\left(1-\frac{2\lambda_a}{1+\lambda_a^2}\lambda_a
+\frac{\lambda_a^2-1}{1+\lambda_a^2}\right)
\end{split}
\end{equation}
and both terms vanish identically.
\end{proof}
\begin{rmk}\label{rmk:trivial_geod}
Equation~\eqref{eq:kernel_coords}, together with~\eqref{eq:realpart_summand}, shows that~$\tilde{\m{M}}$ is affine along trivial geodesics in~$\alpha^{\bb{C}}$. Notice also that~\eqref{eq:trivial_geod} implies that~$\langle\delbar\Im(\dot{\psi}),\del\Re(\dot{\psi})\rangle_{\omega_t}$ coincides with~$\langle\delbar\Re(\dot{\psi}),\del\Re(\dot{\psi})\rangle_{B_t}$, at least if~$B_t$ is a positive form. In this case a trivial geodesic is then given by a parallel pair of K\"ahler geodesics in two different classes.
\end{rmk}
The equations~\eqref{eq:geodesic_coupled} can be rewritten as a system of degenerate-elliptic equations on a manifold with boundary. This is well-known for the first equation in~\eqref{eq:geodesic_coupled}, which is just the geodesic equation in the space of K\"ahler potentials, see~\cite{Donaldson_SymmKahlerHam}. A special case of the second equation in~\eqref{eq:geodesic_coupled}, obtained by setting~$\Im(\dot{\varphi})=0$, has been studied by Collins and Yau in~\cite{CollinsYau_dHYMgeodesics}, where they also showed that their equation can be written as a degenerate dHYM equation. Following their argument, consider the annulus~$A=\set*{z\in\bb{C}\tc\mrm{e}^{-1}<\abs{z}<1}$ and the product~$A\times X$. This is a complex manifold, and we denote by~$D$ and~$\bar{D}$ its Dolbeault operators.
\begin{prop}\label{prop:geodesic_disk}
Fix a complexified K\"ahler form~$\omega_0^{\bb{C}}=B_0+\I\omega_0$ on~$X$. A solution of~\eqref{eq:geodesic_coupled} is equivalent to an~$\bb{S}^1$-invariant function~$\Phi\in\m{C}^\infty(A\times X,\bb{C})$ such that
\begin{equation}\label{eq:geod_coupled_disk}
\begin{dcases}
\left(\pi^*\omega_0+\I D\bar{D}\Im(\Phi)\right)^{n+1}=0\\
\Im\left[\mrm{e}^{-\I\left(\frac{\pi}{2}+\hat{\vartheta}\right)}\left(\pi^*\omega^{\bb{C}}_0+\I D\bar{D}\Phi\right)^{n+1}\right]=0.
\end{dcases}
\end{equation}
Moreover, if~$\Phi$ satisfies~\eqref{eq:geod_coupled_disk} and the corresponding geodesic in~$\alpha^{\bb{C}}$ is nontrivial, the geodesic belongs to~$\m{H}$ if and only if
\begin{equation}
\Re\left[\mrm{e}^{-\I\left(\frac{\pi}{2}+\hat{\vartheta}\right)}\left(\pi^*\omega^{\bb{C}}_0+\I D\bar{D}\Phi\right)^{n+1}\right]>0.
\end{equation}
\end{prop}
\begin{proof}
The equivalence is given by sending a path of potentials~$\set{\varphi_t}\subset\m{C}^\infty(X,\bb{C})$ to the~$\bb{S}^1$-invariant function~$\Phi(z,x)=\varphi_{-\log\abs{z}}(x)$. The proof is a matter of direct computation, essentially already carried out in~\cite{CollinsYau_dHYMgeodesics}, see the proof of Lemma~$2.8$ \textit{ibid}. It is shown that, for any~$\chi\in\m{A}^{1,1}(X,\bb{C})$
\begin{equation}\label{eq:annulus_computations}
\left(\pi^*\chi+\I D\bar{D}\Phi\right)^{n+1}=\frac{(n+1)\I\mrm{d}z\wedge\mrm{d}\bar{z}}{\abs{z}^2}\left(\ddot{\varphi}\,\chi_\varphi^n-n\,\I\del\dot{\varphi}\wedge\delbar\dot{\varphi}\wedge\chi_\varphi^{n-1}\right)
\end{equation}
which readily implies the equivalence between~\eqref{eq:geodesic_coupled} and
\begin{equation}
\begin{dcases}
&\left(\pi^*\omega_0+\I D\bar{D}\Im(\Phi)\right)^{n+1}=0\\
&\Re\left[\mrm{e}^{-\I\hat{\vartheta}}\left(\pi^*\omega_0^{\bb{C}}+\I D\bar{D}\Phi\right)^{n+1}\right]=0.
\end{dcases}
\end{equation}
To obtain~\eqref{eq:geod_coupled_disk}, we just have to notice that
\begin{equation}
\Im\left[\mrm{e}^{-\I\left(\frac{\pi}{2}+\hat{\vartheta}\right)}\left(\pi^*\omega^{\bb{C}}_0+\I D\bar{D}\Phi\right)^{n+1}\right]=-\Re\left[\mrm{e}^{-\I\hat{\vartheta}}\left(\pi^*\omega_0^{\bb{C}}+\I D\bar{D}\Phi\right)^{n+1}\right].
\end{equation}
The second claim now follows from the proof of Proposition~\ref{prop:convex_geod}: the calculations in that proof shows that, along a geodesic,
\begin{equation}
\begin{split}
\Im\big[\mrm{e}^{-\I\hat{\vartheta}}\left(\ddot{\varphi}\,(\omega^{\bb{C}}_\varphi)^n-n\,\I\del\dot{\varphi}\wedge\delbar\dot{\varphi}\wedge(\omega^{\bb{C}}_\varphi)^{n-1}\right)\big]>0
\end{split}
\end{equation}
if and only if~$\cos(\Theta-\vartheta)>0$, i.e.\ the path of potentials belongs to~$\m{H}$. By~\eqref{eq:annulus_computations}, this is in turn equivalent to 
\begin{equation}
\begin{split}
\Im\big[\mrm{e}^{-\I\hat{\vartheta}}\left(\pi^*\omega_0^{\bb{C}}+\I D\bar{D}\Phi\right)^{n+1}\big]>0
\end{split}
\end{equation}
and~$\Re\big[\mrm{e}^{-\I\left(\frac{\pi}{2}+\hat{\vartheta}\right)}\left(\pi^*\omega_0^{\bb{C}}+\I D\bar{D}\Phi\right)^{n+1}\big]=\Im\big[\mrm{e}^{-\I\hat{\vartheta}}\left(\pi^*\omega_0^{\bb{C}}+\I D\bar{D}\Phi\right)^{n+1}\big]$.
\end{proof}
In general, there might not be smooth solutions to~\eqref{eq:geod_coupled_disk}, even for smooth boundary data, since the equation is \emph{degenerate} elliptic. It is well-known however that there always are~$\m{C}^{1,1}$-solutions of the degenerate Monge-Amp\`ere equation with smooth boundary data, see~\cite{Chen_SpaceofKahlerMetrics, TosattiWeinkove_C11regularity}. The same result holds for the degenerate dHYM equation, for a fixed smooth K\"ahler metric in~$\alpha$ and assuming that~$(X,\alpha^{\bb{C}})$ is in the hypercritical range: between any two smooth functions in~$\m{H}$ there is a (unique) path~$\varphi_t\in\m{C}^{1,1}$ that satisfies the dHYM geodesic equation, see~\cite{CollinsYau_dHYMgeodesics, CollinsChuLee_calibrated11forms}. Moreover, this path lies in (a suitable extension of)~$\m{H}$. 

It is natural to conjecture that this regularity should also hold for solutions of the coupled system~\eqref{eq:geod_coupled_disk}, which would then imply the following:
\begin{conj}\label{conj:geodesics_existence}
If~$(X,\alpha^{\bb{C}})$ is hypercritical then for every~$\varphi_0,\varphi_1\in\m{H}$ there is a unique path~$\varphi_t\in\m{H}$ that solves~\eqref{eq:geodesic_coupled}, joins~$\varphi_0$,~$\varphi_1$, and such that~$\varphi_t\in\m{C}^{1,1}(X,\bb{C})$.
\end{conj}
Notice that if the dHYM equation for~$(X,\alpha^{\bb{C}})$ has hypercritical phase, then the corresponding degenerate dHYM equation on~$\m{X}$ has (strictly) supercritical phase and vice-versa. This suggests that it might not be possible to establish the regularity of solutions without assuming the hypercritical phase condition on~$(X,\alpha^{\bb{C}})$.

\begin{rmk}
	Note that if $\omega^{\bb{C}}$ is a complexified K\"ahler form with \emph{hypercritical} phase, then both $\Re(\omega^{\bb{C}})$ and $\Im(\omega^{\bb{C}})$ are K\"ahler forms, cf.\ \cite[Proof of Proposition $5.1$]{JacobYau_special_Lag}.
\end{rmk}

We conclude this section by observing that, in the hypercritical case, the complexified K-energy is also convex along paths of functions that satisfy a more familiar condition than~\eqref{eq:geodesic_coupled}. See also~\cite{KingLeungLee_MomentCoupled} for a similar observation about the K\"ahler-Yang-Mills equations of~\cite{FernandezPradaConsul_coupledKahlerHYM}.
\begin{lemma}\label{lem:Kahler_geod_convex}
Assume that~$(X,\alpha^{\bb{C}})$ has hypercritical phase. Then, the complexified K-energy is convex along paths~$\varphi_t\in\m{H}$ that satisfy
\begin{equation}\label{eq:Kahler_geod_system}
\begin{cases}
\Im(\ddot{\varphi})-\norm{\del\Im(\dot{\varphi})}^2_{\omega_t}=0\\
\Re(\ddot{\varphi})-\norm{\del\Re(\dot{\varphi})}^2_{B_t}=0\\
\end{cases}
\end{equation}
for~$\omega_t=\omega+\I\del\delbar\Im(\varphi)$ and~$B_t=B+\I\del\delbar\Re(\varphi)$. Moreover, the complexified K-energy is strictly convex unless~$\varphi_t$ is a trivial geodesic.
\end{lemma}
The equations in~\eqref{eq:Kahler_geod_system} are simply describing two geodesics of K\"ahler potentials in the classes~$\alpha$ and~$\beta$. While it is well-known that there always exist~$\m{C}^{1,1}$ solutions of~\eqref{eq:Kahler_geod_system} between any pairs of smooth K\"ahler forms, it is not clear if these solutions lie in~$\m{H}$, even assuming that the endpoints do. It would, however, be very useful to use these geodesics, since the existence of~$\m{C}^{1,1}$-solutions would allow us to prove the uniqueness of solutions for~\eqref{eq:scalar_dHYM_singola}, similarly to what is done in Section~\ref{sec:uniqueness} for the case of surfaces. Recall also from Remark \ref{rmk:trivial_geod} that trivial geodesics also solve both~\eqref{eq:geodesic_coupled} and~\eqref{eq:Kahler_geod_system}.
\begin{proof}[Proof of Lemma~\ref{lem:Kahler_geod_convex}]
We proceed as in the proofs of Lemma~\ref{lemma:secondvar} and Proposition~\ref{prop:convex_geod}. Let~$\zeta^1,\dots,\zeta^n$ be a local frame for~$T^*X$ in which~$\omega$ is the standard symplectic form and~$B$ is diagonal, with eigenvalues~$\lambda_1,\dots,\lambda_n$. Let also~$\eta=\Theta(B_t+\I\omega_t)-\hat{\vartheta}$ and~$\varphi(t)=u(t)+\I v(t)$ for real functions~$u,v$ on~$X$. In what follows, we do not indicate the dependence on~$t$ explicitly.

The K\"ahler geodesic equation~\eqref{eq:Kahler_geod_system} implies that, with our new notation,
\begin{equation}
\begin{cases}
\ddot{v}=\sum\abs{(\del\dot{v})_a}^2\\
\ddot{u}=\sum\lambda_a^{-1}\abs{(\del\dot{u})_a}^2
\end{cases}
\end{equation}
and so we obtain
\begin{equation}
\mrm{e}^{-\I\hat{\vartheta}}\ddot{\varphi}\,(\omega^{\bb{C}}_\varphi)^n=\omega^nr(\omega^{\bb{C}})\mrm{e}^{\I\eta}\sum\left(\I\abs{(\del\dot{v})_a}^2+\lambda_a^{-1}\abs{(\del\dot{u})_a}^2\right).
\end{equation}
Together with~\eqref{eq:complex_forms}, this gives
\begin{equation}
\begin{split}
&\mrm{e}^{-\I\hat{\vartheta}}\ddot{\varphi}\,(\omega^{\bb{C}}_\varphi)^n-n\,\I\del\dot{\varphi}\wedge\delbar\dot{\varphi}\wedge\mrm{e}^{-\I\hat{\vartheta}}(\omega^{\bb{C}}_\varphi)^{n-1}=\\
=\omega^nr(\omega^{\bb{C}})\mrm{e}^{\I\eta}\sum\Bigg[&\I\abs{(\del\dot{v})_a}^2+\lambda_a^{-1}\abs{(\del\dot{u})_a}^2\\&-\frac{\abs{(\del\dot{u})_a}^2-\abs{(\del\dot{v})_a}^2+\I\left((\del\dot{u})_a(\delbar\dot{v})_a+(\del\dot{v})_a(\delbar\dot{u})_a\right)}{1+\lambda_a^2}(-\I+\lambda_a)\Bigg].
\end{split}
\end{equation}
A straightforward computation then shows
\begin{equation}\label{eq:Im_convexity_Kahlergeod}
\begin{split}
\Im&\big[\mrm{e}^{-\I\hat{\vartheta}}\left(\ddot{\varphi}\,(\omega^{\bb{C}}_\varphi)^n-n\,\I\del\dot{\varphi}\wedge\delbar\dot{\varphi}\wedge(\omega^{\bb{C}}_\varphi)^{n-1}\right)\big]=\\
&=\omega^nr(\omega^{\bb{C}})\sum_{a=1}^n\frac{\abs{(\del\dot{u})_a-\lambda_a(\del\dot{v})_a}^2}{1+\lambda_a^2}(\cos\eta+\lambda_a^{-1}\sin\eta).
\end{split}
\end{equation}
We claim that if the phase is hypercritical and~$u+\I v$ is an almost calibrated potential, for any ~$a=1,\dots,n$
\begin{equation}\label{eq:convex_positivity_Kahlergeod}
\cos\eta+\lambda_a^{-1}\sin\eta\geq 0.
\end{equation}
Our hypotheses imply that there is a lift~$\hat{\vartheta}\in(0,\pi/2)$, and~$-\pi/2<\eta-\hat{\vartheta}<\pi/2$. Also, as we are assuming~$B_u$ to be a K\"ahler metric, the eigenvalues~$\lambda_a$ are all positive. Then~\eqref{eq:convex_positivity_Kahlergeod} is equivalent to~$\tan\eta>-\lambda_a$, which is to say~$\eta>-\arctan\lambda_a$, or~$\eta+\pi/2>\arccot\lambda_a$. Unravelling the definition of~$\eta$ then~\eqref{eq:convex_positivity_Kahlergeod} is equivalent to
\begin{equation}
\sum_{b\not=a}\arccot\lambda_b>\hat{\vartheta}-\frac{\pi}{2}
\end{equation}
and~$\hat{\vartheta}<\pi/2$. To conclude the proof, it is sufficient to notice that~\eqref{eq:secondvar_geod} together with~\eqref{eq:Im_convexity_Kahlergeod} and~\eqref{eq:convex_positivity_Kahlergeod} imply that the second variation of~$\tilde{\m{M}}$ is positive along~$\varphi(t)$, and strictly positive unless~$\dot{v}$ is a holomorphy potential and~$\delbar \dot{u}=\nabla^{1,0}_{\omega_v}\dot{v}\lrcorner B$.
\end{proof}

\section{Uniqueness of solutions}\label{sec:uniqueness}

The question of the regularity of geodesics is of fundamental importance to establish the uniqueness of solutions of~\eqref{eq:scalar_dHYM}. In some cases, it might happen that one can show that two K\"ahler metrics are joined by a smooth geodesic of K\"ahler potentials. This is the case for example if one considers torus-invariant K\"ahler metrics on a toric manifold. We now show that in these cases uniqueness of solutions can be established directly from our formula for the second variation of the energy functional and Remark~\ref{rmk:convex_along_dHYM}. We assume that $\m{H}\not=\emptyset$, as in the previous section.
\begin{prop}
	Let~$\omega_0^{\bb{C}}$ and~$\omega_1^{\bb{C}}$ be two almost calibrated solutions of~\eqref{eq:scalar_dHYM} in the same complexified K\"ahler class $\alpha^{\bb{C}}$. Assume moreover that there is a \emph{smooth} geodesic of K\"ahler potentials joining~$\omega_0$ and~$\omega_1$ and that~$(X,\alpha^{\bb{C}})$ is supercritical. Then there exists~$f\in\mrm{Aut}_{\mrm{red}}(X)$ such that~$f^*\omega_1^{\bb{C}}=\omega_0^{\bb{C}}$.
\end{prop}
One way to prove this would be to use Collins-Yau's theory of geodesics for the dHYM equation, in the hyper-critical case. But the existence and uniqueness of solutions of dHYM coming from Theorem~\ref{thm:dHYM_stability} allow us to give a more direct argument.
\begin{proof}
Let~$\set{v_t}\subset\m{C}^\infty(X,\bb{R})$ be a smooth geodesic of K\"ahler potentials such that the corresponding path of K\"ahler forms~$\omega_t\coloneqq\omega_0+\I\del\delbar v_t$ joins~$\omega_0$ and~$\omega_1$. As we are assuming that the phase is supercritical, by Theorem~\ref{thm:dHYM_stability} there is a corresponding path of functions~$u_t$ such that~$B_t\coloneqq B+\I\del\delbar\,u_t$ solves~$\Im\left(\mrm{e}^{-\I\hat{\vartheta}}(B_t+\I\omega_t)\right)=0$. This path is smooth, by Corollary~\ref{cor:smoothdependence}. Then the second variation of~$\tilde{\m{M}}$ along~$\varphi_t\coloneqq u_t+\I v_t$ is convex, and as both endpoints of~$\varphi_t$ are critical points of~$\tilde{\m{M}}$, the second variation of~$\tilde{\m{M}}(\varphi_t)$ must vanish identically. So~$\dot{v}$ is a holomorphy potential with respect to~$\omega_v$, and~$\dot{u}$ satisfies~$\nabla^{1,0}_{\omega_v}\dot{v}\lrcorner B_u=\I\delbar\dot{u}$ for every~$t$.

The vector fields~$-\frac{1}{2}\nabla_{\omega_v}\dot{v}_t$ are all holomorphic, and their isotopy~$f_t$ is a family of biholomorphisms in~$\mrm{Aut}_{\mrm{red}}(X)$, by definition. Moreover,~$f_t$ conjugates~$\omega_v$ and~$\omega_0$, so by the uniqueness of solutions of the dHYM equation we find~$f_t^*B_u=B_0$, and in particular~$f_1^*(B_1+\I\omega_1)=B_0+\I\omega_0$.
\end{proof}
As a consequence, we obtain a uniqueness result on toric manifolds. In symplectic coordinates on a toric K\"ahler manifold, a K\"ahler geodesic simply corresponds to a line of symplectic potentials \cite[Proposition $2.2$]{SongZelditch_geodesics}. In particular, the K\"ahler geodesic between two smooth toric K\"ahler metrics is smooth.
\begin{prop}
	Let~$X$ be a toric manifold, and let~$\alpha^{\bb{C}}\coloneqq\beta+\I\alpha$ be a torus-invariant complexified K\"ahler class with supercritical phase. Then, the torus-invariant solutions of~\eqref{eq:scalar_dHYM_singola} in~$\alpha^{\bb{C}}$ are unique up to the action of~$\mrm{Aut}_{\mrm{red}}(X)$.
\end{prop}

\subsection{Complex surfaces}

If~$X$ is a complex surface, the coupled system~\eqref{eq:scalar_dHYM} can be simplified substantially by a change of variables. If we define
\begin{equation}
\chi=\sin\hat{\vartheta}\,B-\cos\hat{\vartheta}\,\omega
\end{equation}
and rescale the coupling constant~$\gamma$ to~$\tilde{\gamma}=\abs{\gamma}/\sin^2\hat{\vartheta}$, then~\eqref{eq:scalar_dHYM} is equivalent to
\begin{equation}\label{eq:scalar_dHYM_surface}
\begin{dcases}
\chi^2=\omega^2\\
s(\omega)=\tilde{\gamma}\Lambda_\omega\chi+\tilde{c}
\end{dcases}
\end{equation}
where~$\tilde{c}$ is a constant, determined through integration by~$[\omega]$,~$[\chi]$, and~$\tilde{\gamma}$. Note that the definition of~$\vartheta$ implies~$\int_X\chi^2=\alpha^2$, so that the first equation in~\eqref{eq:scalar_dHYM_surface} can be solved if~$[\chi]$ or~$-[\chi]$ is a K\"ahler class, by Yau's solution of the Calabi Conjecture. In~\cite{JacobYau_special_Lag} Jacob and Yau proved that this condition is necessary and sufficient for the dHYM equation~$\chi^2=\omega^2$ to have solutions, for any fixed K\"ahler form~$\omega\in\alpha$. We will only consider the case~$[\chi]>0$, so that the second equation in~\eqref{eq:scalar_dHYM_surface} is a twisted cscK equation.

The system~\eqref{eq:scalar_dHYM_surface} for~$\tilde{\gamma}=1$ is a particular case of the \emph{coupled cscK system} considered in~\cite{DatarPingali_coupledcscK}, which in turn generalizes the \emph{coupled K\"ahler-Einstein equations} of~\cite{HultgrenWittNystrom_coupledKE}. Similarly to the general dimensional situation considered in Section~\ref{sec:complexified_Kenergy},~\eqref{eq:scalar_dHYM_surface} is the Euler-Lagrange equation of a functional defined on~$[\omega]\times[\chi]$ by the condition
\begin{equation}
D\m{M}'_{(\omega,\chi)}(\dot{u},\dot{v})=-\tilde{\gamma}\int_X\dot{v}(\chi^2-\omega^2)-\int_X\dot{u}\left(s(\omega)-\tilde{c}-\tilde{\gamma}\Lambda_{\omega}\chi\right)\omega^2.
\end{equation}
As is the case for~$\tilde{\m{M}}$, the second variation of~$\m{M}'$ around a critical point is positive semi-definite. In general, the second variation of~$\m{M}'$ along a path~$(u_t,v_t)$ is
\begin{equation}
\begin{split}
\partial^2_t\m{M}'(u_t,v_t)=&\tilde{\gamma}\int_X\Big(\norm*{\partial\dot{v}}_{\chi_v}^2-\ddot{v}-\norm*{\partial\dot{u}}_{\chi_v}^2\Big){\chi_v}^2+\tilde{\gamma}\int_X\Big(\ddot{v}-\left\langle\mrm{d}\dot{v},\mrm{d}\dot{u}\right\rangle_{\omega_u}+\ddot{u}\,\Lambda_{\omega_u}{\chi_v}\Big){\omega_u}^2\\
&+\norm*{\m{D}_{\omega_u}\,\dot{u}}^2_{L^2({\omega_u})}-\int_X\left(\ddot{u}-\norm*{\del\dot{u}}^2_{\omega_u}\right)(s(\omega_u)-\tilde{c}){\omega_u}^2.
\end{split}
\end{equation}
As both~$[\omega]$ and~$[\chi]$ are K\"ahler classes, it makes sense to consider paths~$(u_t,v_t)$ that satisfy
\begin{equation}\label{eq:surface_geod}
\begin{cases}
\ddot{u}=\norm*{\del\dot{u}}_{\omega_u}^2,\\
\ddot{v}=\norm*{\del\dot{v}}_{\chi_v}^2.
\end{cases}
\end{equation}
\begin{prop}
The functional~$\m{M}'$ is convex along paths~$(u_t,v_t)$ that satisfy the geodesic equations~\eqref{eq:surface_geod}, and it is affine along trivial geodesics (c.f. Section~\ref{sec:geod_equations}).
\end{prop}
\begin{proof}
If~$u_t$ and~$v_t$ are K\"ahler geodesics, we have
\begin{equation}
\begin{split}
\partial^2_t\m{M}'(u_t,v_t)=&\norm*{\m{D}_{\omega_u}\,\dot{u}}^2_{L^2({\omega_u})}+\\
&+\tilde{\gamma}\int_X\left[\Big(\norm*{\del\dot{v}}_{\chi_v}^2-\left\langle\mrm{d}\dot{v},\mrm{d}\dot{u}\right\rangle_{\omega_u}+\norm*{\del\dot{u}}_{\omega_u}^2\,\Lambda_{\omega_u}{\chi_v}\Big){\omega_u}^2-\norm*{\partial\dot{u}}_{\chi_v}^2\,{\chi_v}^2\right].
\end{split}
\end{equation}
We will show that the integrand of the~$\tilde{\gamma}$ factor is (semi-)positive. As in the proof of Proposition~\ref{prop:convex_geod}, fix a coframe for~$TX$ in which~$\omega_u$ is the canonical symplectic form and~$\chi_v$ is diagonal, with eigenvalues~$\lambda_1$,~$\lambda_2$. Then the integrand is
\begin{equation}
\begin{gathered}
\lambda_1^{-1}\abs*{\del\dot{v}_1}^2+\lambda_2^{-1}\abs*{\del\dot{v}_2}^2-\left(\del\dot{u}_1\,\overline{\del\dot{v}}_1+\del\dot{v}_1\,\overline{\del\dot{u}}_1+\del\dot{u}_2\,\overline{\del\dot{v}}_2+\del\dot{v}_2\,\overline{\del\dot{u}}_2\right)\\
\quad+\left(\abs{\del\dot{u}_1}^2+\abs{\del\dot{u}_2}^2\right)(\lambda_1+\lambda_2)-\left(\lambda_2\abs*{\del\dot{u}_1}^2+\lambda_1\abs*{\del\dot{u}_2}^2\right)=\\
=\frac{\abs*{\del\dot{v}_1-\lambda_1\del\dot{u}_1}^2}{\lambda_1}+\frac{\abs*{\del\dot{v}_2-\lambda_2\del\dot{u}_2}^2}{\lambda_2}.
\end{gathered}
\end{equation}
So~$\partial^2_t\m{M}'(u_t,v_t)$ is always semi-positive, and strictly positive unless~$\nabla^{1,0}_{\omega_u}\dot{u}$ is holomorphic and~$\nabla^{1,0}_{\omega_u}\dot{u}=\nabla^{1,0}_{\chi_v}\dot{v}$.
\end{proof}
We can use this convexity result to prove the uniqueness of solutions of the scalar curvature equation with B-field on surfaces, in the case when~$[\chi]$ is a K\"ahler class, up to the action of the reduced automorphism group of~$X$.

For the remainder of this section, let~$\tilde{\beta}\coloneqq[\chi]$ and assume that~$\alpha$,~$\tilde{\beta}$ are K\"ahler classes such that~$\alpha^2=\tilde{\beta}^2$. Fixing reference metrics~$\omega_0\in\alpha$ and~$\chi_0\in\tilde{\beta}$, we can write
\begin{equation}
\begin{split}
\m{M}'(u,v)=&\int\log\left(\frac{\omega_u^2}{\omega_0^2}\right)\omega_u^2+\frac{\tilde{c}}{3}\sum_{p=0}^2\int u\,\omega_0^p\wedge\omega_u^{2-p}-\int u\,(\omega_0+\omega_u)\wedge\mrm{Ric}(\omega_0)\\
&+\tilde{\gamma}\left(\int v\,\omega_u^2-\frac{1}{3}\sum_{p=0}^2\int v\,\chi_0^p\wedge\chi_v^{2-p}+\int u\,(\omega_0+\omega_u)\wedge\chi_0\right)
\end{split}
\end{equation}
or more compactly as
\begin{equation}\label{eq:Kenergy_surf_ChenTian}
\m{M}'(u,v)=\m{H}^\alpha(u)+\frac{\tilde{c}}{3}\m{E}^\alpha(u)-\m{E}^\alpha_{\mrm{Ric}(\omega_0)}(u)-\tilde{\gamma}\left(\frac{1}{3}\m{E}^{\tilde{\beta}}(v)-\m{E}^\alpha_{\chi_0}(u)-\int v\,\omega_u^2\right)
\end{equation}
where~$\m{H}^\alpha$ and~$\m{E}^\alpha$ are the entropy and (twisted) energy functionals of the K\"ahler class~$\alpha$, while~$\m{E}^{\tilde{\beta}}$ is the energy functional of the class~$\tilde{\beta}$.

Equation~\eqref{eq:Kenergy_surf_ChenTian} shows that~$\m{M}'(u,v)$ is well-defined even if~$u$ and~$v$ are not smooth, as long as~$\omega+\I\del\delbar u$ and~$\chi+\I\del\delbar v$ are positive currents with $L^\infty$ coefficients, similarly to what happens for the classical K-energy, cf.\ \cite[\S$3$]{Chen_lowerbounds}. Hence, analogously to~\cite[Theorem~$1.1$]{BermanBerndtsson_uniquness}, we have the following.
\begin{prop}\label{prop:convexity_weakgeod}
	For any two K\"ahler classes~$\alpha$,~$\tilde{\beta}$ on~$X$, the energy functional~$\m{M}'$ is convex along weak geodesic segments in~$\alpha\times\tilde{\beta}$, i.e.\ paths~$u_t,v_t\in\m{C}^{1,1}(X)$ satisfying~\eqref{eq:surface_geod}. In particular, if there is a solution to equation~\eqref{eq:scalar_dHYM_surface}, then~$\m{M}'$ is bounded below on~$\alpha\times\tilde{\beta}$.
\end{prop}
Indeed, the main point of the proof of~\cite[Theorem~$1.1$]{BermanBerndtsson_uniquness} is to show that the entropy part of the K-energy functional is continuous along geodesics of K\"ahler potentials (a priori, it is only lower semi-continuous). Our functional~$\m{M}'$ is the sum of the K-energy and a generalised energy functional which is convex along weak geodesics, so the same argument as~\cite[\S$3.1.1$]{BermanBerndtsson_uniquness} gives Proposition~\ref{prop:convexity_weakgeod}.
\begin{rmk}
It is natural to conjecture that an analogue of~\ref{prop:convexity_weakgeod} should also hold for complexified K\"ahler classes on higher-dimensional manifolds. In fact, if we knew the existence of~$\m{C}^{1,1}$ complexified geodesics (so if Conjecture~\ref{conj:geodesics_existence} holds), the same proof would show that the complexified K-energy functional~\eqref{eq:Kenergy_decomposition} is convex along weak geodesics in the space of almost calibrated representatives, and if there is a solution of~\eqref{eq:scalar_dHYM_singola}, then~$\tilde{\m{M}}$ is bounded below on this space.
\end{rmk}
Our main goal in this section is to show that, with minor changes, it is also possible to adapt the proof from~\cite{BermanBerndtsson_uniquness} to prove the uniqueness of solutions of~\eqref{eq:scalar_dHYM_surface}.
\begin{thm}\label{thm:surface_uniqueness}
Assume that~$\alpha$ and~$\tilde{\beta}$ are K\"ahler classes on~$X$ such that~$\alpha^2=\tilde{\beta}^2$, and that~$(\omega_0,\chi_0)$ and~$(\omega_1,\chi_1)$ are solutions of~\eqref{eq:surface_geod} in~$\alpha\times\tilde{\beta}$. Then there is~$\phi\in\mrm{Aut}_{\mrm{red}}(X)$ such that~$\phi^*(\omega_1,\chi_1)=(\omega_0,\chi_0)$.
\end{thm}
We can follow the argument in~\cite{BermanBerndtsson_uniquness} almost directly, the main difference is that, while the~$K$-energy is defined on the space of K\"ahler forms in~$\alpha$, our functional~$\m{M}'$ is defined on pairs~$(\omega,\chi)$ of K\"ahler forms in~$\alpha\times\tilde{\beta}$, so we will need to make some minor adjustments to the original proof of Berman and Berndtsson.

The main ingredients of the proof are the convexity properties of~$\m{M}'$ along weak geodesics, and the positivity of its Hessian. By this we mean the Hessian of~$\m{M}'$ with respect to the Donaldson-Mabuchi connection on~$\alpha\times\tilde{\beta}$, which a straightforward computation shows to be
\begin{equation}\label{eq:Hessian_M'-symmetric}
\begin{split}
\mrm{Hess}\,&\m{M}_{(\omega,\chi)}\left((u_0,v_0),(u_1,v_1)\right)=\int\Re\left(g_\omega(\m{D} u_0,\overline{\m{D}}u_1)\right)\omega^2+\\
&+\frac{\tilde{\gamma}}{2}\int\left[g_\chi(\mrm{d}v_0,\mrm{d}v_1)+g_\chi(\mrm{grad}_\omega u_0,\mrm{grad}_\omega u_1)-g_\omega(\mrm{d}v_0,\mrm{d}u_1)-g_\omega(\mrm{d}u_0,\mrm{d}v_1)\right]\omega^2.
\end{split}
\end{equation} 
Alternatively, integrating by parts we can express the Hessian by
\begin{equation}\label{eq:Hessian_M'-operator}
\begin{split}
\mrm{Hess}\,&\m{M}'_{(\omega,\chi)}\left((u_0,v_0),(u_1,v_1)\right)=\int u_1\,\Re\left(\m{D}^*\m{D} u_0\right)\omega^2+\\
+&\frac{\tilde{\gamma}}{2}\int u_1\,\mrm{div}_\omega\left(T^\omega_\chi(\mrm{grad}_\chi v_0-\mrm{grad}_\omega u_0)\right)\omega^2+\frac{\tilde{\gamma}}{2}\int v_1\mrm{div}_\omega\left(\mrm{grad}_\chi v_0-\mrm{grad}_\omega u_0\right)\omega^2,
\end{split}
\end{equation}
where~$T^\omega_\chi\coloneqq\omega^{-1}\chi$; more explicitly,~$T^\omega_\chi$ is the endomorphism of~$TX$ defined by composing the isomorphisms defined by~$\omega$ and~$\chi$,~$TX\xrightarrow{\chi}T^*X\xrightarrow{\omega}TX$.

We let~$Q:\m{C}^\infty(X)\times\m{C}^\infty(X)\to \m{C}^\infty(X)\times\m{C}^\infty(X)$ be the operator
\begin{equation}
Q(u,v)=\left(
\Re\left(\m{D}^*\m{D} u\right)+\frac{\tilde{\gamma}}{2}\mrm{div}_\omega\left(T^\omega_\chi(\mrm{grad}_\chi v_0-\mrm{grad}_\omega u_0)\right),
\frac{\tilde{\gamma}}{2}\mrm{div}_\omega\left(\mrm{grad}_\chi v-\mrm{grad}_\omega u\right)\right),
\end{equation}
which is elliptic, and self-adjoint with respect to the~$L^2(\omega)$ pairing on~$\m{C}^\infty(X)\times\m{C}^\infty(X)$. Equation~\eqref{eq:Hessian_M'-operator} is thus equivalent to
\begin{equation}
\mrm{Hess}\,\m{M}'_{(\omega,\chi)}\left((u_0,v_0),(u_1,v_1)\right)=\left\langle Q(u_0,v_0),(u_1,v_1) \right\rangle_{L^2(\omega)}.
\end{equation}
Proceeding as in the proof of Proposition~\ref{prop:convex_geod}, it is easy to check that~$\mrm{Hess}\,\m{M}'_{(\omega,\chi)}$ is semipositive-definite, and the isotropic directions are given by the tangent vectors~$(u,v)$ that lie in the kernel of~$Q$, i.e.~$\mrm{grad}_\omega u$ is real holomorphic and~$\mrm{grad}_\omega u=\mrm{grad}_\chi v$.
\begin{proof}[Proof of Theorem~\ref{thm:surface_uniqueness}]
Assume that~$(\omega_0,\chi_0)$ and~$(\omega_1,\chi_1)$ are solutions of~\eqref{eq:scalar_dHYM_surface}. Using~$\omega_0$ and~$\chi_0$ as reference points in~$\alpha$ and~$\tilde{\beta}$, we identify the sets of K\"ahler forms in~$\alpha$ and~$\tilde{\beta}$ with subsets of~$\m{C}^\infty(X)$, through the~$\del\delbar$-Lemma.

Let~$\mu$ be a volume form on~$X$ such that~$\int_X\mu=\alpha^2=\tilde{\beta}^2$, and define a function~$\m{F}$ on~$\alpha\times\tilde{\beta}$ as
\begin{equation}
\m{F}_\mu(u,v)=\int u\,\mu-\frac{1}{3}\m{E}^{\alpha}(u)+\frac{\tilde{\gamma}}{2}\left(\int v\,\mu-\frac{1}{3}\m{E}^{\tilde{\beta}}(v)\right).
\end{equation}
Notice that the differential of~$\m{F}_\mu$ does not depend on the choice of a reference point in~$\alpha\times\tilde{\beta}$, and it can be written as
\begin{equation}
D\m{F}_\mu\Bigr|_{(u,v)}(\dot{u},\dot{v})=\int\dot{u}(\mu-\omega_u^2)+\frac{\tilde{\gamma}}{2}\int\dot{v}(\mu-\chi_v^2).
\end{equation}
In particular, as~$\chi_i^2=\omega_i^2$,
\begin{equation}
D\m{F}_\mu\Bigr|_{(\omega_i,\chi_i)}(\dot{u},\dot{v})=\int\left(\dot{u}+\frac{\tilde{\gamma}}{2}\dot{v}\right)(\mu-\omega_i^2).
\end{equation}

Recall from Remark~\ref{rmk:kernel} that the kernel of the linearisation of~\eqref{eq:scalar_dHYM_surface} (that is, the kernel of~$Q$) around a solution is in bijective correspondence with~$\f{h}_0$, the set of holomorphic vector fields on~$X$ that have a holomorphy potential. Let~$\mrm{Aut}_0(X)\coloneqq\mrm{exp}(\f{h}_0)\subseteq\mrm{Aut}(X)$ be the corresponding subgroup of biholomorphisms of the surface, and for~$i=0,1$ denote by~$S_i\subset\alpha\times\tilde{\beta}$ the set of all pairs of K\"ahler metrics obtained by pulling-back~$(\omega_i,\chi_i)$ by an element of~$\mrm{Aut}_0(X)$. By definition, any two points of~$S_i$ can be joined by a (trivial, smooth) geodesic, and as~$\m{F}_\mu$ is strictly convex along geodesics,~\cite[Proposition~$4.7$]{BermanBerndtsson_uniquness} shows that there exists a minimum~$(\hat{\omega}_i,\hat{\chi}_i)\in S_i$. In particular,
\begin{equation}\label{eq:Fmu_nucleodifferenziale}
D\m{F}_\mu\Bigr|_{(\hat{\omega}_i,\hat{\chi}_i)}(\dot{u},\dot{v})=0\mbox{ for every }(\dot{u},\dot{v})\in\mrm{ker}\,Q.
\end{equation}
We will show that~$(\hat{\omega}_0,\hat{\chi}_0)=(\hat{\omega}_1,\hat{\chi}_1)$. The key point is that~\eqref{eq:Fmu_nucleodifferenziale} implies the existence of tangent vectors~$(\dot{u}_i,\dot{v}_i)\in T_{(\hat{\omega}_i,\hat{\chi}_i)}\alpha\times\tilde{\beta}$ for~$i=0,1$ such that 
\begin{equation}\label{eq:differenziale_nullo}
\left(D_{(\dot{u}_i,\dot{v}_i)}(D\m{M}')\right)\Bigr|_{(\hat{\omega}_i,\hat{\chi}_i)}+D\m{F}_\mu\Bigr|_{(\hat{\omega}_i,\hat{\chi}_i)}=0.
\end{equation}
To show this, say for~$i=0$, notice that~\eqref{eq:differenziale_nullo} is equivalent to
\begin{equation}\label{eq:Hessian_vanishingQ}
Q(\dot{u}_0,\dot{v}_0)=-\left(\frac{\mu-\hat{\omega}_0^2}{\hat{\omega}_0^2},\frac{\tilde{\gamma}}{2}\frac{\mu-\hat{\omega}_0^2}{\hat{\omega}_0^2}\right).
\end{equation}
By the Fredholm alternative,~\eqref{eq:Hessian_vanishingQ} has a solution if and only if~$\left(\frac{\mu-\hat{\omega}_0^2}{\hat{\omega}_0^2},\frac{\tilde{\gamma}}{2}\frac{\mu-\hat{\omega}_0^2}{\hat{\omega}_0^2}\right)$ lies in the~$L^2(\hat{\omega}_0)$-orthogonal of the kernel of~$Q$, and this property is exactly~\eqref{eq:Fmu_nucleodifferenziale}.

Having established~\eqref{eq:differenziale_nullo}, we proceed as in the proof of~\cite[Theorem~$4.4$]{BermanBerndtsson_uniquness}; define a modified version of~$\m{M}'$ as
\begin{equation}
\m{M}_s'(u,v)=\m{M}'(u,v)+s\,\m{F}_\mu(u,v)
\end{equation}
and consider, for very small~$s$ and for~$i=0,1$, the differential of~$\m{M}'_s$ at the two points~$(\omega_i(s),\chi_i(s))\coloneqq(\hat{\omega}_i+\I\del\delbar\,s\dot{u}_i,\hat{\chi}_i+\I\del\delbar\,s\dot{v}_i)$. As~$(\omega_i(0),\chi_i(0))=(\hat{\omega}_i,\hat{\chi}_i)$ are critical points of~$\m{M}'$, this differential satisfies (c.f.~\eqref{eq:differenziale_nullo})
\begin{equation}
\partial_{s=0}\left(D\m{M}'_s\Bigr|_{(\omega_i(s),\chi_i(s))}\right)=\left(D_{(\dot{u}_i,\dot{v}_i)}D\m{M}'\right)\Bigr|_{(\hat{\omega}_i,\hat{\chi}_i)}+D\m{F}\Bigr|_{(\hat{\omega}_i,\hat{\chi}_i)}=0,
\end{equation}
so that~$(D\m{M}_s')_{(\omega_i(s),\chi_i(s))}=O(s^2)$. At this point, the argument goes exactly as in the proof of~\cite[Theorem~$4.4$]{BermanBerndtsson_uniquness}, and we refer to that paper for the details. Briefly, what we showed up to this point is that~$(\omega_i(s),\chi_i(s))$ is almost a critical point for~$\m{M}'_s$. Connect~$(\omega_0(s),\chi_0(s))$ and~$(\omega_1(s),\chi_1(s))$ by a weak geodesic segment~$(u_s(t),v_s(t))$;~$\m{F}_\mu$ is strictly convex along weak geodesics, and quantifiably so (\cite[Proposition~$4.1$]{BermanBerndtsson_uniquness}). Hence~$\m{M}'_s$ is also strictly convex along~$(u_s(t),v_s(t))$, which gives an upper bound of~$O(s)$ for the distance between~$(\omega_0(s),\chi_0(s))$ and~$(\omega_1(s),\chi_1(s))$. Taking the limit for~$s\to 0$ we deduce~$(\hat{\omega}_0,\hat{\chi}_0)=(\hat{\omega}_1,\hat{\chi}_1)$, so~$S_0=S_1$.
\end{proof}

\addcontentsline{toc}{section}{References}
\printbibliography

\end{document}